\documentclass[preprint,11pt]{elsarticle}

\usepackage[english]{babel}
\usepackage{amsmath}
\usepackage{mathrsfs}
\usepackage{amssymb,latexsym}
\usepackage{amsthm}
\usepackage{enumerate}
\usepackage{ifthen}
\usepackage{bbm}
\usepackage{color}
\usepackage{xcolor}
\usepackage{graphicx}
\provideboolean{shownotes} 
\setboolean{shownotes}{true}
\usepackage{hyperref}

\usepackage{scalerel}
\usepackage[usestackEOL]{stackengine}

\usepackage{esint}

\newcommand{\margnote}[1]{
\ifthenelse{\boolean{shownotes}}%
{\marginpar{\raggedright\tiny\texttt{#1}}}%
{}%
}

\newcommand{\hole}[1]{
\ifthenelse{\boolean{shownotes}}%
{ \begin{center}\fbox{ \rule {.25cm}{0cm}
\rule[-.1cm]{0cm}{.4cm} \parbox{.85\textwidth}{
\texttt{#1}} \rule {.25cm}{0cm}}\end{center}}
{}
}

\def\dashint{\,\ThisStyle{\ensurestackMath{%
            \stackinset{c}{.2\LMpt}{c}{.5\LMpt}{\SavedStyle-}{\SavedStyle\phantom{\int}}}%
        \setbox0=\hbox{$\SavedStyle\int\,$}\kern-\wd0}\int}

\newtheorem{theorem}{Theorem}[section]
\newtheorem{lemma}[theorem]{Lemma}
\newtheorem{proposition}[theorem]{Proposition}

\theoremstyle{remark}
\newtheorem{remark}[theorem]{Remark}
\newtheorem{definition}[theorem]{Definition}

\numberwithin{equation}{section}

\DeclareMathOperator{\dive}{div}

\newcommand{\e}{\epsilon}

\newcommand{\R}{\mathbb{R}}
\newcommand{\T}{\mathbb{T}^{3}}
\newcommand{\MT}{\mathcal{T}}
\newcommand{\rrho}{\sqrt{\rho}}

\newcommand{\re}{\rho_{\e}}
\newcommand{\rre}{\sqrt{\rho_{\e}}}

\newcommand{\ue}{u_{\e}}

\newcommand{\weakto}{\rightharpoonup}

\def\dashint{\,\ThisStyle{\ensurestackMath{%
            \stackinset{c}{.2\LMpt}{c}{.5\LMpt}{\SavedStyle-}{\SavedStyle\phantom{\int}}}%
        \setbox0=\hbox{$\SavedStyle\int\,$}\kern-\wd0}\int}

\newenvironment{acknowledgments}{%
  \bigskip
  \noindent\textbf{Acknowledgments.}\quad
}{\par}

\begin{document}

\begin{frontmatter}

\title{Global Weak Solutions for the High-Friction Quantum Navier–Stokes–Poisson Model}

\author[univaq]{Giada Cianfarani Carnevale}
\address[univaq]{Dipartimento di Ingegneria Industriale e dell’Informazione e di Economia, and Dipartimento di Ingegneria e Scienze dell’Informazione e Matematica, Universit\`a degli Studi dell'Aquila (Italy)}
\ead{giada.cianfaranicarnevale@univaq.it}

\begin{abstract}
In \cite{ACLS}, the Authors rigorously establish the relaxation limit from the Quantum Navier–Stokes–Poisson (QNSP) system to the Quantum Drift–Diffusion (QDD) equation, while providing only a brief outline of the global existence theory for weak solutions to QNSP in the high-friction regime (see Appendix A therein). In this manuscript, we present a complete and fully self-contained proof of global existence.

More precisely, we prove the global existence of finite-energy weak solutions to the QNSP system with high friction and large initial data on the three-dimensional torus $\mathbb{T}^3$. The model describes a compressible, viscous quantum fluid with Korteweg-type capillarity effects, and allows for degenerate viscosity and vacuum regions.

The construction proceeds in two main steps. First, it is introduced a Faedo–Galerkin approximation endowed with suitable damping mechanisms, which yields smooth approximate solutions through compactness arguments. Then, it will be justify the convergence of the approximating sequence by combining a truncation of the momentum equation with DiPerna–Lions commutator estimates, providing the required control over the nonlinear transport structure.
\end{abstract}


\end{frontmatter}

\section{Introduction}\label{sec:0}
The aim of the present manuscript is to develop a complete and rigorous construction of finite-energy weak solutions to Quantum Navier–-Stokes–-Poisson (QNSP) system, that is

\begin{equation}\label{qns1APP} 
\left\{
\begin{aligned}
& \partial_t \rho + \mathrm{div}(\rho u) = 0, \\[4pt]
& \partial_t(\rho u)
  + \mathrm{div}(\rho u \otimes u)
  - \mathrm{div}(\rho\, D u)
  + \nabla \rho^{\gamma}
  + \rho \nabla V
  = 2\rho \nabla\!\left(
      \frac{\Delta \sqrt{\rho}}{\sqrt{\rho}}
    \right) - \rho u, \\[4pt]
& -\Delta V = \rho - g.
\end{aligned}
\right. 
\end{equation}
Our investigations are motivated by the analysis in \cite{ACLS}, where the Authors present a rigorous study of the diffusive relaxation limit  for weak solutions to the
QNSP system in the high-friction regime, while the global existence theory of such solutions is only sketched in Appendix A.

Precisely, we recall that the Authors in \cite{ACLS} consider the following compressible quantum hydrodynamic system on the three-dimensional torus $\mathbb{T}^3$:
\begin{equation}\label{eq:Corrado}
\left\{
\begin{aligned}
& \partial_t \rho_\varepsilon + \frac{1}{\epsilon}\, \mathrm{div}(\rho_\varepsilon u_\varepsilon)=0, \\
& \partial_t(\rho_\varepsilon u_\varepsilon)
 + \frac{1}{\epsilon}\mathrm{div}(\rho_\varepsilon u_\varepsilon \otimes u_\varepsilon)
 - \frac{1}{\epsilon}\mathrm{div}(\rho_\varepsilon D u_\varepsilon)
 + \frac{1}{\epsilon}\nabla \rho_\varepsilon^{\gamma}
 + \frac{1}{\varepsilon}\rho_\varepsilon \nabla V_\varepsilon \\
& \qquad\qquad= \frac{1}{\epsilon}2\rho_\varepsilon \nabla\!\left(
      \frac{\Delta \sqrt{\rho_\varepsilon}}{\sqrt{\rho_\varepsilon}}
    \right)
 - \frac{1}{\epsilon^2}\rho_\varepsilon u_\varepsilon, \\
& -\Delta V_\varepsilon = \rho_\varepsilon - g .
\end{aligned}
\right.
\end{equation}

Here, $\rho_\varepsilon, u_\varepsilon, D u_\varepsilon, V_\varepsilon$ denote the particle density,
the velocity field, the symmetric gradient of $u_\varepsilon$, and the electrostatic potential,
respectively, while $g$ represents the doping profile and $\varepsilon>0$ is the high-friction
parameter.

In \cite{ACLS} the Authors prove the \emph{weak--weak convergence} of solutions to \eqref{eq:Corrado} in the
high-friction limit, towards the \emph{Quantum Drift--Diffusion equation}, that is
\begin{equation*}
    \begin{aligned}
        &\partial_t \rho + \dive\left(2 \rho \nabla \left( \frac{\Delta \sqrt{\rho}}{\sqrt{\rho}}\right) - \nabla \rho^{\gamma} - \rho \nabla V\right) =0\\
        &- \Delta V = \rho - g.
    \end{aligned}
\end{equation*} 
This result provides an important relaxation framework in quantum hydrodynamics and serves as a motivation for the
analysis developed here.

Such solutions play a fundamental role in the analysis of complex quantum fluid models, especially in the presence of diffusive and nonlinear effects. Although two principal approximation strategies for system \eqref{qns1APP} are available in the literature—one by Antonelli--Spirito \cite{AS1} and one by Lacroix-Violet--Vasseur \cite{LV}—the problem of developing alternative approximation schemes that lead to weak solutions is still open.

The approximation scheme by Antonelli and Spirito \cite{AS1}, although effective, imposes technical restrictions on the physical parameters of the system. In particular, the pressure exponent must satisfy $\gamma \in (1,3)$ and the viscosity $\nu$ must be comparable to the quantum diffusion $k$, i.e.\ $\nu \sim k$. In contrast, the scheme introduced by Lacroix-Violet and Vasseur \cite{LV} offers greater flexibility: it imposes no restrictions on the coefficients and provides a more general framework for constructing weak solutions. For this reason, we primarily follow the latter method.

We now discuss the main difficulties in the analysis of the original Quantum Navier--Stokes system, that is
\begin{equation}\label{qnsAPP}
\left\{
\begin{aligned} 
& \partial_t\rho + \mathrm{div}(\rho u ) = 0, \\ 
& \partial_t(\rho u ) +\mathrm{div} (\rho u \otimes u) - \nu\,\mathrm{div} (\rho D(u)) + \nabla \rho^{\gamma} 
= \,\rho \nabla\!\left( \frac{\Delta \sqrt{\rho}}{\sqrt{\rho}} \right).
\end{aligned}
\right.
\end{equation}
The first difficulty arises from the presence of the third-order tensor in the momentum equation, the so-called \emph{Bohm potential}. As recalled from \cite{BGLV} it holds the following identity,
\[
\mathrm{div}(\rho \nabla^2 \log \rho)
  = \rho \nabla\!\left( \frac{\Delta \sqrt{\rho}}{\sqrt{\rho}} \right)
  = 2\sqrt{\rho}\,\nabla^2 \sqrt{\rho} - 2 \nabla \sqrt{\rho} \otimes \nabla \sqrt{\rho},
\]
where no positivity or nonnegativity of the density is known \emph{a priori}. A possible approach is the vanishing-viscosity method \cite{Feir2}, which allows one to recover a strong maximum principle and obtain a positive lower bound for the density, depending on the $L^\infty$ norm of $\mathrm{div}\,u$.

A second difficulty lies in the density-dependent viscosity $\nu \rho$, which degenerates at vacuum. Indeed, in vacuum regions the velocity field $u$ and its gradient can not be defined appropriately, while the momentum $m =\rho u$ is always meaningful. However, the compactness of $\sqrt{\rho}u$ is not guaranteed for system \eqref{qnsAPP} due to the lack of suitable a priori estimates. The classical energy inequality, 
does not provide the necessary compactness for $\rho^\gamma$ and the quantum term $\nabla\sqrt{\rho}$. This problem is overcome by introducing the Bresch--Desjardins entropy \cite{BD, BD1, BD3} associated to \eqref{qnsAPP}, which provides additional regularity for the density:
\[
\begin{aligned}
\frac{d}{dt}\!\int_{\mathbb{T}^3}
\frac{1}{2}\rho w^2 
+ \frac{\rho^{\gamma}}{\gamma - 1}
+ 2|\nabla \sqrt{\rho}|^2 
+ (\rho(\log \rho -1) + 1)\,dx \\
+ \int_{\mathbb{T}^3} \rho |A(u)|^2\,dx
+ \frac{4}{\gamma}\!\int_{\mathbb{T}^3} |\nabla \rho^{\gamma/2}|^2\,dx
+ \int_{\mathbb{T}^3} \rho |\nabla^2 \log \rho|^2\,dx
\leq 0,
\end{aligned}
\]
where $w:=u + \nabla \log \rho$.

The existence of weak solutions for system \eqref{qnsAPP} with degenerate viscosity but without quantum effects ($k=0$) was proved in \cite{vass.yu, LX, LM}, where compactness of $\sqrt{\rho}u$ is ensured by the \emph{Mellet--Vasseur estimate}, namely the boundedness of $\rho |u|^2\log (1+|u|^2)$ in $L^{\infty}(0,T; L^1(\mathbb{T}^3))$. For $k>0$ such an estimate does not hold in general.

Antonelli and Spirito \cite{AS3} introduced an alternative formulation of the QNS system in terms of the effective velocity $w=u + c\nabla \log \rho$, showing that the constant $c$ can be chosen depending on $\nu$ and $k$ so that the dispersive term disappears. Their subsequent work \cite{AS} shows that a Mellet--Vasseur estimate can be recovered under additional structural conditions:  
in 2D when $k<\nu$ and $\gamma>1$, and in 3D when $k^2<\nu^2<9k^2/8$ and $\gamma\in (1,3)$.  
A global existence result for $k>\nu$ and $\gamma>3$, without vacuum regions, was obtained in \cite{Jun_SIAM}.  
In the multidimensional case, the existence of global strong solutions is proved in \cite{Haspot} under the assumption $k=\nu$ and linear pressure.

In the present  work we consider the relaxation parameter fixed $\epsilon=1$ and the aforementioned issues are overcome via two approximation procedures. The first step is very similar to the strategy of \cite{Feireisl, Vasseur, Feir2, Z}. In particular, in Section \ref{sec:1} we establish the global existence of weak solutions for the Quantum Navier--Stokes--Poisson system augmented with artificial friction terms $-r_0 u - r_1\rho |u|^2u$, with $r_0,r_1>0$, $\nu=1$. The system becomes
\begin{equation}\label{eq:qnskapp}
\left\{
\begin{aligned} 
& \partial_t\rho + \mathrm{div}(\rho u )=0, \\ 
& \partial_t(\rho u)
 +  \mathrm{div}(\rho u \otimes u)
 - \mathrm{div} (\rho Du)
 + \nabla \rho^{\gamma}
 + \rho \nabla V \\
&\qquad\qquad=
 \mathrm{div}(\rho \nabla^2 \log \rho)
 - \rho u - r_0 u - r_1\rho|u|^2u,\\
& - \Delta V = \rho - g.
\end{aligned}
\right.
\end{equation}
We start with a Faedo--Galerkin scheme and proceed with classical compactness arguments. We underline that the term $r_0u$ ensures the velocity $u$ is defined a.e.\ in $(0,T)\times\mathbb{T}^3$, while the term $r_1\rho |u|^2u$ plays a crucial role in establishing compactness of the convective term $\sqrt{\rho}u$ (see Lemma~\ref{lemma 3.3}).

The main theorem of Section~\ref{sec:1} is the following:
\begin{theorem}\label{thm1app}
Let $(\rho,u,V)$ be a finite-energy weak solution of \eqref{2.6} with $\rho_0,u_0,V_0,g$ satisfying \eqref{initdata}, \eqref{initdatamom}, \eqref{meanVapp}, \eqref{compcondapp}. Then, as $\varepsilon,\mu,\delta,\eta \rightarrow 0$, such solutions converge to a finite-energy weak solution of \eqref{sys1app} in the sense of Definition~\ref{def1app}.
\end{theorem}

The second step, carried out in Section~\ref{sec2}, follows the arguments present in \cite{AS, LV}. Namely we will study the vanishing limits as $r_0,r_1:=\varepsilon$ of the  system
\begin{equation}\label{eq:qnseapp}
\left\{
\begin{aligned} 
&\partial_t\rho_\varepsilon + \mathrm{div}(\rho_\varepsilon u_\varepsilon)=0,\\ 
&\partial_t(\rho_\varepsilon u_\varepsilon)
 + \mathrm{div} (\rho_\varepsilon u_\varepsilon\otimes u_\varepsilon)
 - \mathrm{div} (\rho_\varepsilon D u_\varepsilon)
 + \nabla \rho_\varepsilon^{\gamma}
 + \rho_\varepsilon \nabla V_\varepsilon \\
&\qquad\qquad=
 \mathrm{div}(\rho_\varepsilon \nabla^2 \log \rho_\varepsilon)
 - \rho_\varepsilon u_\varepsilon
 - \varepsilon u_\varepsilon
 - \varepsilon \rho_\varepsilon |u_\varepsilon|^2 u_\varepsilon,\\
& - \Delta V_\varepsilon = \rho_\varepsilon - g_\varepsilon,
\end{aligned}
\right.
\end{equation}
and we will prove that as $\varepsilon \rightarrow 0$, the solutions $(\rho_\varepsilon,u_\varepsilon,V_\varepsilon)$ of \eqref{eq:qnseapp} converge to weak solutions of \eqref{qns1APP}.

In this part of the approximation, the classical energy and BD entropy are not sufficient to control the convective term. The key idea is to adopt a slightly different notion of weak solution, namely \emph{renormalized weak solutions}, and perform a truncation argument as DiPerna--Lions in \cite{DPL}. The essential improvement lies in gaining the required velocity regularity by truncating near vacuum. For a full exposition of renormalized solutions we refer to \cite{LV}, and for the classical truncation argument to \cite{DPL}.

Similarly, in \cite{AS} Antonelli and Spirito employ two distinct truncations, acting respectively on the momentum equation and on the density. In our framework, however, the improved regularity available for the density allows us to dispense with the latter, only the truncation of the momentum equation is needed.

The main theorem of Section~\ref{sec2} is the following:
\begin{theorem}\label{teo:3iapp}
Let $(\rho,u,V)$ be a finite-energy weak solution of \eqref{1.8} in the sense of Definition~\ref{def1app}. Suppose $\rho^0$, $\rho^0 u^0$, $V$, and $g$ satisfy \eqref{initdata1}, \eqref{meanVapp}, \eqref{compcondapp}. Then, as $\varepsilon \rightarrow 0$, the sequence converges to a finite-energy weak solution $(\rho, u, V)$ of \eqref{1.1} in the sense of Definition~\ref{def:1}.
\end{theorem}

We recall that global existence for weak solutions for Quantum Hydrodynamic systems has been proved in \cite{AM, AM_CM} without restrictions on initial data, while weak solutions in the far-field case appear in \cite{AHS}. Nonuniqueness via convex integration is obtained in \cite{DFM}. Specifically, for the quantum Navier--Stokes--Poisson system the global existence and algebraic decay estimates were obtained by Tong and Xia \cite{Tong_Xia_22}, while Wu analyzed the large-time asymptotic behavior in the bipolar case \cite{Wu_22}. Additional stability results were established by Wu and Hou \cite{Wu_Hou_23}, and for outflow problems by Wu and Zhu \cite{Wu_Zhu_23}. In the one-dimensional framework, Liu and Zhang proved global well-posedness with large initial data \cite{Liu_Zhang_24}. More recently, Chen and Zhao obtained global strong solutions together with the vanishing dispersion limit for the one-dimensional quantum Navier--Stokes equations \cite{Chen_Zhao_25}. 
Recent contributions include the analysis of one-dimensional fourth-order evolution equations arising as gradient flows of the Korteweg energy, the Authors in \cite{GS} prove the global-in-time existence of non-negative weak solutions without imposing any upper bound on the exponent appearing in the energy functional. While, in the stochastic setting, was proved in \cite{Pescatore_23} the existence of weak martingale solutions to the stochastic 1D Quantum-Navier-Stokes equations. These contributions represent the current state of the art in the mathematical analysis of quantum hydrodynamic models.

\section{Existence of global weak solutions to Quantum Navier-Stokes system \texorpdfstring{\eqref{sys1app}}{sys1app}}\label{sec:1}

The goal of this Section is to prove Theorem \ref{thm1app}. The proof consists of several steps and it is based on the well-known papers \cite{Feireisl} and \cite{Vasseur}.

Let us present the Cauchy problem defined on $(0,T) \times \T$ we are going to study in this section. For a given function $g:\T \rightarrow \R$ we have
\begin{equation}\label{sys1app}
\left\{\begin{aligned} 
& \partial_t\rho + \dive(\rho u ) = 0 \\ 
& \partial_t(\rho u) + \dive(\rho u \otimes u) - \dive (\rho Du) + \nabla \rho^{\gamma} +  \rho \nabla V = \\
& \dive(\rho \nabla^2 \log \rho) -  \rho u - r_0u - r_1\rho|u|^2u\\
&- \Delta V = \rho - g
\end{aligned} \right.
\end{equation}
with initial data
\begin{equation}\label{initdata1}
\begin{aligned}
&\rho(0,x)= \rho^0(x), \; (\rho u)(0,x)= m^0(x), \; \text{ such that } m^0 = 0 \text{ if } \rho^0 = 0, \\
& -\log_{-} \rho(0,x)= -\log_{-} \rho^0(x)= \log_+ \rho^0(x) = \log \min \{\rho^0, 1 \}, \\
\end{aligned}
\end{equation}
and $V$, $g \in L^2(\T)$ such that the following conditions hold
\begin{equation}\label{meanVapp}
\dashint_{\T}{V}(x,t)dx = 0,
\end{equation}
\begin{equation}\label{compcondapp}
\dashint_{\T} g(x)dx = M, \text{ where } M:=\int_{\T} \rho^0(x)dx
\end{equation}
The definition of weak solutions for system $\eqref{sys1app}$ is the following.
\begin{definition}\label{def1app}
Given $\rho_0$ positive and such that $\sqrt{\rho_0} \in H^{1}(\T)$ and $\rho_0\in L^{\gamma}(\T)$, $g \in L^2(\T)$ and $u_0$ such that $\sqrt{\rho_0}u_0 \in L^{2}(\T)$, then $(\rho, u,  V)$ with $\rho \geq 0$ and $V$ with zero average is a weak solution of the Cauchy problem \eqref{sys1app}-\eqref{initdata1} if the following conditions are satisfied:	
\begin{itemize}
\item Integrability condition:
\begin{equation*}
\begin{aligned}
& \sqrt{\rho} \in L^{\infty}((0,T);H^{1}(\T))\cap L^{2}((0,T);H^2(\T)) \\
& \sqrt{\rho} u \in L^{\infty}((0,T);L^{2}(\T))\\
&\rho^{\gamma}\in L^{\infty}((0,T);L^1(\T)) \qquad \rho \in C((0,T);L^2(\T))\\ & V \in C((0,T);H^{2}(\T)) \qquad \MT \in L^{2}((0,T);L^{2}(\T))
\end{aligned}
\end{equation*}
\item Continuity Equation holds for any $\phi \in C^{\infty}([0,T) \times \T; \R)$ such that $\phi(T)=0$:
\begin{equation}\label{cont1}
\int_{\T} \rho^0(x) \phi(0,x)dx + \iint_{(0,T)\times \T} \rho \partial_t \phi +  \rho u \cdot \nabla \phi dxdt = 0
\end{equation}
\item Momentum Equation holds for any $\psi \in C^{\infty}([0,T) \times \T; \R^3)$ such that $\psi(T)=0$:
\begin{equation}\label{mom1}
\begin{aligned}
&\int_{\T} m^0(x)\,\psi(0,x)\,dx 
    + \iint_{(0,T)\times\T} \rho u\,\partial_t \psi\,dxdt
     \\
     &+ \iint_{(0,T)\times\T} \rho u\otimes u : \nabla \psi dxdt \\
& - \iint_{(0,T)\times\T} \rho\, D(u) : \nabla \psi\,dxdt
    + \iint_{(0,T)\times\T} \nabla \rho^{\gamma}\, \psi\,dxdt
     \\
     & + \iint_{(0,T)\times\T} \rho\, \nabla V\, \psi\,dxdt + \iint_{(0,T)\times\T} \rho u\,\psi\,dxdt \\
&  - 2 \iint_{(0,T)\times\T} \sqrt{\rho}\,\nabla^2 \sqrt{\rho} : \nabla\psi\,dxdt
    + 2 \iint_{(0,T)\times\T} \nabla \sqrt{\rho}\otimes\nabla \sqrt{\rho}\,\psi\,dxdt \\
& 
    + r_0 \iint_{(0,T)\times\T} u\,\psi\,dxdt
    + r_1 \iint_{(0,T)\times\T} \rho |u|^2 u\,\psi\,dxdt
= 0.
\end{aligned}
\end{equation}
\item Poisson Equation:
\begin{equation}\label{poi}
-\Delta V = \rho - g \text{ for  a.e in } (0,T)\times \T.
\end{equation}
\item Energy dissipation:
\begin{equation}\label{enedissapp}
\begin{aligned}
 \iint_{(0,T)\times\T}\sqrt{\rho}\MT\phi\,dxdt
 =&    -\iint_{(0,T)\times\T}\rho u\nabla\phi\,dxdt \\&
  -\iint_{(0,T)\times\T}2\sqrt{\rho}u_i\otimes\nabla_j\sqrt{\rho}\phi\,dxdt.
\end{aligned}
\end{equation}
for any $\phi \in  C^{\infty}((0,T) \times \T; \R)$
\item Energy inequality:\\
For a.e $t \in (0,T)$
\begin{equation}\label{ene1}
\begin{aligned}
&\int_{\T}  \frac{1}{2} \rho u^2 + \frac{\rho^{\gamma}}{\gamma - 1} + 2 |\nabla \sqrt{\rho}|^2 + \frac{1}{2} |\nabla V|^2(t)dx   \\
& + \iint_{(0,t) \times \T} \rho |D(u)|^2dxdt + \iint_{(0,t) \times \T} r_1 \rho |u|^4 dxdt \\
& + \iint_{(0,t) \times \T}  r_0 |u|^2 dxdt \\
& \leq \int_{\T}  \frac{1}{2} \frac{(m_0)^2}{\rho_0} + \frac{(\rho_0)^{\gamma}}{\gamma - 1} + 2 |\nabla \sqrt{\rho_0}|^2 + \frac{1}{2} |\nabla V_0|^2dx
\end{aligned}
\end{equation}
\item BD entropy: \\ Let $w = u + \nabla \log \rho$, there exists a tensor $\mathcal{S} \in L^2((0,T)\times \T$ such that we have for a.e $t \in (0,T)$:
\begin{equation}\label{defsapp}
    \sqrt{\rho}S=2 \sqrt{\rho}\nabla^2\sqrt{\rho} - 2 \nabla \sqrt{\rho} \otimes \nabla \sqrt{\rho} \text{ a.e in } (0,T)\times \T,
\end{equation}
\begin{equation}\label{bd1}
\begin{aligned}
&\int_{\T} \frac{1}{2}\rho |w|^2 + \frac{\rho^{\gamma}}{\gamma -1} + 2|\nabla\sqrt{\rho}|^2 - r_0\log \rho dx  \\
& + \int_{\T} (\rho(\log \rho -1)+1) + \frac{1}{2}|\nabla V|^2(t)dx\\ 
&+ \iint_{(0,t)\times\T}\rho |A(u)|^2 dx ds
+ \frac{4}{\gamma}\iint_{(0,t)\times\T}|\nabla \rho^{\gamma/2}|^2 dx ds \\
&+ r_0\iint_{(0,t)\times\T}|u|^2 dx ds + r_1\iint_{(0,t)\times\T}\rho |u|^4 dx ds \\
& + \iint_{(0,t)\times\T}\rho |u|^2 dx ds
+ \iint_{(0,t)\times\T}\rho |\nabla^2 \log \rho|^2 dx ds \\
&+ \iint_{(0,t)\times\T}\rho(\rho - g) dx ds \\
& \leq
\int_{\T} \frac{1}{2}\rho_0 |w_0|^2 + \frac{\rho_0^{\gamma}}{\gamma -1} + 2|\nabla\sqrt{\rho_0}|^2 - r_0\log \rho_0 dx \\
&+ \int_{\T}(\rho_0(\log \rho_0 -1)+1) dx + \frac{1}{2}|\nabla V_0|^2 dx
\end{aligned}
\end{equation}
where $C$ is bounded by the initial energy and $-\log_{-}\rho = \log \min(\rho,1)$.
\item There exists an absolute constant $C>0$ such that:
\begin{equation}\label{qpart}
\begin{aligned}
& \int_{(0,t) \times \T} |\nabla^2\sqrt{\rho}|^2 + |\nabla \rho^{1/4}|^4 dxds \leq \\
& C \int_{\T} \frac{1}{2} \rho_0 |w_0|^2 + \frac{\rho_0^{\gamma}}{\gamma -1} + 2 |\nabla\sqrt{\rho_0}|^2 dx\\
& + C\int_{\T}(\rho_0(\log \rho_0 -1)+1) + \frac{1}{2}|\nabla V_0|^2 - r_0 \log \rho_0 dx
\end{aligned}
\end{equation}
\end{itemize}
\end{definition}

\begin{remark}[\emph{About the Bohm Potential}]\label{rem:1-1}
As we have already mentioned in the Section \ref{sec:0} the quantum term can be written in different ways

\begin{equation*}
2\rho \nabla\left( \frac{\Delta \sqrt{ \rho}}{\sqrt{\rho}}\right)= \dive (\rho \nabla^2 \log \rho) = \dive ( 2 \sqrt{\rho} \nabla^2 \sqrt{\rho} - 2 \nabla\sqrt{\rho}\otimes \nabla \sqrt{\rho}).
\end{equation*} 
In the forthcoming approximation scheme (see subsection \ref{FG}), due to the improved regularity of the function $\rho$ one can use indifferently the above expressions in the weak formulation. On the other hand, in Section \ref{sec2}, the compactness achieved for the function $\rho$ allow us to use only the last expression. The same happens, for example, in the relaxation limit of the original system $\eqref{qns1APP}$, see \cite{ACLS} and therein references.
\end{remark}
\begin{remark}[\emph{About the Poisson Equation}]
The compatibility condition $\eqref{compcondapp}$ comes from the conservation of the average of the density. Indeed, from the continuity equation and the positivity of $\rho$ it is known that there exists $M>0$ such that $\int_{\T} \rho(t,x)\,dx= M = \int_{\T} \rho^0(x)$ for any $t\in(0,T)$. Therefore $\rho - g$ is chosen to be periodic on $\T$ for every $t \in [0,T]$, and togheter with $\eqref{meanVapp}$ the well-posedness for the Poisson equation $\eqref{sys1app}_3$ is guaranteed.
\end{remark}
\begin{remark}[\emph{About the initial data}]
Concerning the initial data, the uniform integrability of $\sqrt{\rho^0} \in H^1(\T)$, $\rho^0 \in L^{\gamma}(\T)$ and $\log_+ \rho^0 \in L^1(\T)$, implies:
\begin{equation*}
\rho^0 \log \rho^0 \text{ is uniformly bounded in } L^1(\T),
\end{equation*}	
without other restrictions.
\end{remark}
 
\subsection{Approximate problem and a priori estimates} 
In order to prove Theorem \ref{thm1app} one has to introduce the following modified system on $(0,T) \times \T$: 
\begin{equation*}
\left\{\begin{aligned} 
& \partial_t\rho + \dive(\rho u ) = \epsilon \Delta\rho \\ 
& \partial_t(\rho u) + \dive(\rho u \otimes u) - \dive (\rho Du) + \nabla \rho^{\gamma} - \eta \nabla \rho^{-10} - \mu \Delta^2 u + \epsilon \nabla \rho \cdot \nabla u \\ 
& \; \;  +  \rho \nabla V =   \dive(\rho \nabla^2 \log \rho) - \rho u - r_0u - r_1\rho|u|^2u + \delta \rho \nabla \Delta^9 \rho\\
& - \Delta V = \rho - g 
\end{aligned} \right.
\end{equation*}
with data satisfying $\eqref{meanVapp}$, $\eqref{compcondapp}$ and 
\begin{equation}\label{initdata}
\rho(0,x)= \rho^0(x) \in C^{2+\nu}(\T), \; 0< \underline{\rho} < \rho^0(x) < \overline{ \rho},
\end{equation}
\begin{equation}\label{initdatamom}
 (\rho u)(0,x)= m^0(x) \; \in C^{2}(\T),
\end{equation}
where $\nu >0$ and $\epsilon>0, \mu>0, \delta>0$, $\eta>0$ are chosen to be small. Let us underline that the presence of the terms $\eta \nabla \rho^{-10}$ and $\delta \rho \nabla \Delta^9 \rho$ has the desired regularizing effect on the integrability properties of $\rho$. In subsection \ref{subsec2} we will show that $\nabla \log \rho$ becomes an admissible test function and, consequently, the computation of the BD entropy associated to $\eqref{2.6}$ is rigorously justified. The term $\epsilon \nabla \rho \cdot \nabla u$ is introduced in order to preserve the positivity of the energy inequality associated to the modified system. The boundedness of $\mu\Delta u$ in $L^2((0,T) \times \T)$ will be the key property to prove the compactness of the gradient of the velocity field for the approximating sequence.

\begin{remark}
    The solution depends on the parameters $\e,\mu,\delta,\eta$ and one has $(\rho_{\e,\mu,\delta,\eta},u_{\e,\mu,\delta,\eta},V_{\e,\mu,\delta,\eta})$. However, to simplify notation, we omit the subscripts. Also, the doping profile $g$ may depend on the parameters, but we confine ourselves to the case $g_{\e,\mu,\delta,\eta}=g$ for any $\e,\mu,\delta,\eta>0$. 
\end{remark}

\begin{remark}[\emph{About the definition of weak solutions}]
The definition of weak solutions for system $\eqref{2.6}$ is slightly different from the one in Definition \ref{def1app}. Due to the presence of the artificial viscosities term in the momentum equation, the regularity of $\rho$ implies that the continuity equation is satisfied pointwise: 
\begin{equation*}
\partial_t\rho + \dive(\rho u) = \e \Delta \rho \qquad \text{ holds for a.e. } \quad (t,x) \in (0,T) \times \T.
\end{equation*} 
\end{remark}

The first result of this section is the following.
\begin{proposition}\label{1 approx}
	There exists a finite energy weak solution $(\rho,u,V)$ to the following system:
	\begin{equation}\label{2.6}
	\left\{\begin{aligned} 
	& \partial_t\rho + \dive(\rho u ) = \epsilon \Delta\rho \\ 
	& \partial_t(\rho u) + \dive(\rho u \otimes u) - \dive (\rho Du) + \nabla \rho^{\gamma} + \rho \nabla V - \eta \nabla \rho^{-10} - \mu \Delta^2 u + \\ 
	& \; \;  \epsilon \nabla \rho \cdot \nabla u =  \dive(\rho\nabla^2 \log \rho) - \rho u - r_0u - r_1\rho|u|^2u + \delta \rho \nabla \Delta^9 \rho,\\
	& - \Delta V = \rho - g
	\end{aligned} \right.
	\end{equation}
	with data satisfying \eqref{meanVapp}, \eqref{compcondapp}, \eqref{initdata}, \eqref{initdatamom}. Moreover, $(\rho,u,V)$ satisfies the energy inequality for a.e. $t \in (0,T)$:
	\begin{equation}\label{limit energy1}
	\begin{aligned}
	&E(t) + \iint_{(0,t)\times \T}\rho |Du|^2\,dxds + \frac{4 \e}{\gamma} \iint_{(0,t)\times \T}|\nabla \rho^{\gamma/2}|^2\,dxds \\
	& + \frac{2}{5}\e\eta \iint_{(0,t)\times \T}|\nabla \rho^{-5}|^2\,dxds + \delta\e \iint_{(0,t)\times \T}|\Delta^5 \rho|^2\,dxds \\
	& + r_0 \iint_{(0,t)\times \T}|u|^2\,dxds + r_1 \iint_{(0,t)\times \T}\rho |u|^4\,dxds + \mu \iint_{(0,t)\times \T}|\Delta u|^2 dxds \\
	&+ \e \iint_{(0,t)\times \T}\rho|\nabla^2 \log \rho|^2\,dxds + \iint_{(0,t)\times \T}\rho|u|^2\,dxds \leq E(0)=E_0
\end{aligned}
	\end{equation}
	where
	\begin{equation*}
	E(t)= \int_{\T} \frac{1}{2}\rho|u|^2 + \eta \frac{\rho^{-10}}{11} + \frac{\rho^{\gamma}}{\gamma - 1} + \delta\frac{|\nabla \Delta^4 \rho|^2}{2} + 2 |\nabla \sqrt{\rho}|^2 + \frac{1}{2}|\nabla V|^2dx.
	\end{equation*}
	and the BD entropy, where $w=u + \nabla \log \rho$:
	\begin{equation}\label{bdinit}
	\begin{aligned}
&\int_{\T} 
    \frac{1}{2}\,\rho |w|^2 
    + \frac{\rho^{\gamma}}{\gamma - 1} 
    + \frac{1}{2}|\nabla \sqrt{\rho}|^2 
    + \eta\, \frac{2}{5}\, \frac{\rho^{-10}}{11} dx \\
    & + \int_{\T}  [\rho(\log \rho -1)+1] 
    - r_0 \log \rho 
    + \delta \frac{|\nabla \Delta^4 \rho|^2}{2}dx \\
& + \iint_{(0,t)\times\T} \rho |A(u)|^2\,dxds
    + \frac{4\varepsilon}{\gamma} \iint_{(0,t)\times\T} |\nabla \rho^{\gamma/2}|^2\,dxds
    \\
&+ \frac{4}{\gamma} \iint_{(0,t)\times\T} |\nabla \rho^{\gamma/2}|^2\,dxds
    + r_0 \int_{\T} \rho u^2
    + r_1 \int_{\T} \rho u^4
    + \iint_{(0,t)\times\T} \rho u^2\,dxds \\
&+ \varepsilon \iint_{(0,t)\times\T} \rho |\nabla \log\rho|^2\,dxds
    + \frac{2\varepsilon\eta}{5} \iint_{(0,t)\times\T} |\nabla \rho^{-5}|^2\,dxds \\
    & + \frac{2\eta}{5} \iint_{(0,t)\times\T} |\nabla \rho^{-5}|^2\,dxds  + \mu \iint_{(0,t)\times\T} |\Delta u|^2\,dxds \\
    & + \delta\varepsilon \iint_{(0,t)\times\T} |\Delta^5\rho|^2dxds \\
    & + \delta \iint_{(0,t)\times\T} |\Delta^5\rho|^2\,dxds 
    + \varepsilon \iint_{(0,t)\times\T} \frac{|\Delta \rho|^2}{\rho}\,dxds \\
    & + \iint_{(0,t)\times\T} \rho |\nabla^2 \log\rho|^2\,dxds
    + \varepsilon \iint_{(0,t)\times\T} \rho |\nabla^2 \log\rho|^2\,dxds
\;\leq\; \\[4pt]
&\int_{\T} 
    \frac{1}{2} \rho^0 |w^0|^2
    + \frac{(\rho^0)^{\gamma}}{\gamma - 1}
    + \frac{1}{2}|\nabla \sqrt{\rho^0}|^2
    + \eta\, \frac{2}{5}\,\frac{(\rho^0)^{-10}}{11}
    + [\rho^0(\log \rho^0 -1)+1] dx\\
&+ \int_{\T} - r_0 \log \rho^0+  \delta \frac{|\nabla \Delta^4 \rho^0|^2}{2} dx
    + \sum_{i=1}^{4} R_i \\
    &
    + \varepsilon\, \|\rho\|_{H^2(\T)}\, \|\rho^{-1}\|_{L^{\infty}(\T)}
    + \tilde{C} E_0.
\end{aligned}
\end{equation}
\end{proposition}

The proof of Proposition \ref{1 approx} consists of two main steps:
\begin{enumerate}
    \item \textbf{Construction of the finite-dimensional approximate problem:} \\
    The continuity equation \eqref{2.6}$_1$ is solved directly, while the momentum equation \eqref{2.6}$_2$ is treated via a fixed-point argument based on a finite-dimensional Faedo--Galerkin approximation scheme. The coupled resolution of these two equations also yields a finite-dimensional formulation of the associated Poisson equation \eqref{2.6}$_3$.
    \item \textbf{Passage to the limit as \( N \to \infty \):} \\
    Once existence is established for the finite-dimensional system, one passes to the limit \( N \to \infty \) using compactness arguments. This allows the extraction of a convergent subsequence whose limit provides a solution to the original problem in the infinite-dimensional setting.
\end{enumerate}

To solve \eqref{2.6}$_1$ with initial data \eqref{initdata}, we recall a result from the classical theory of parabolic equations. For completeness, we state the result here and refer to \cite{Lunardi} for further details.

\begin{lemma}
Assume $u$ is a given vector function belonging to the class
\begin{equation*}
u \in C([0,T];[C^2(\T)]^3).
\end{equation*}
Then the initial boundary value problem $\eqref{2.6}_1$, $\eqref{initdata}$ possesses a unique classical solution $\rho$ on the set $[0,T] \times \T$ such that $\rho(t) \in C^{2+\nu}(\T)$ for any fixed $t \in [0,T]$.
\end{lemma}

\begin{lemma}\label{prop S}
Let the initial datum satisfy $\eqref{initdata}$. Then there exists a mapping $\mathcal{S} =\mathcal{S}(u)$, 
\begin{equation*}
\mathcal{S}:  C([0,T];[C^2(\T)]^3) \rightarrow C([0,T];[C^{2+\nu}(\T)]^3),
\end{equation*} 
enjoying the following properties:
\begin{itemize}
	\item[i)] $\rho = \mathcal{S}(u)$ is the unique classical solution of $\eqref{2.6}_1$, $\eqref{initdata}$;
	\item[ii)] $\underline{\rho} \exp{-\int_0^t ||\dive u||_{L^{\infty}(\T)}ds} \leq \mathcal{S}(u)(t,x) \leq \bar{\rho}\exp{\int_0^t ||\dive u||_{L^{\infty}(\T)}ds}$ for all $t \geq 0$;
	\item[iii)] 
	\begin{equation}\label{S} \hspace{+ 0.5cm}
	||\mathcal{S}(u_1)- \mathcal{S}(u_2)||_{C([0,T];W^{1,2}(\T))} \leq Tc(k,T)||u_1- u_2||_{C([0,T];W_0^{1,2}(\T))}
	\end{equation} 
	for any $u_1,u_2$ belonging to the set
	\begin{small} 
	\begin{equation*}
	 M_k= \left\lbrace u \in C([0,T];W_0^{1,2}(\T)) \; / \; ||u||_{L^{\infty}(\T)} + ||\nabla u||_{L^{\infty}(\T)} \leq k \text{ for all } t \geq 0 \right\rbrace
	\end{equation*}
    \end{small}
\end{itemize}
\end{lemma}

\subsubsection{Faedo-Galerkin scheme}\label{FG}
In order to solve $\eqref{2.6}_2$ one has to introduce a finite dimensional space 
\[
X_N= \left\lbrace e_1,e_2, \cdots, e_N \right\rbrace,\quad N \in \mathbb{N},
\]
where $\{ e_i \}_{i \in \mathbb{N}}$ is an orthonormal basis of $L^2$ and an orthogonal basis of $H^1$. The approximate solutions $u_N \in C([0,T];X_N)$ are constructed in the following way:
\begin{equation*}
u_N = \sum_{i=1}^{N}\lambda_i(t) e_i(x) \qquad (t,x) \in [0,T] \times \T,
\end{equation*}
for some functions $\lambda_i(t) \in C[0,T]$. Thus the velocity $u_N$ can be bounded in $C^0([0,T];C^k(\T))$ since for any $k \geq 0$ it holds:
\begin{equation*}
||u_N||_{C([0,T];C^k(\T))} \leq C ||u_N||_{C([0,T];L^2(\T))},
\end{equation*}
where $C$ is independent of $N$. 

The approximate solution $u_N$ has to satisfy the weak formulation for any test function $\phi \in X_N$:
\begin{equation}\label{mom}
\begin{aligned}
& \int_{\T} \rho u_N (T)\phi\,dx - \int_{\T} m_0 \phi\,dx 
+ \mu \iint_{(0,t)\times \T}\Delta u_N \Delta \phi\,dxds  \\
& - \iint_{(0,t)\times \T} \rho u_N \otimes u_N : \nabla \phi\,dxds + \iint_{(0,t)\times \T} \rho Du_N : \nabla \phi\,dxds \\
& - \iint_{(0,t)\times \T} \rho^\gamma \nabla \phi\,dxds + \eta \iint_{(0,t)\times \T}\rho^{-10} \nabla \phi\,dxds \\
&+ \epsilon \iint_{(0,t)\times \T} \nabla \rho \nabla u_N \phi\,dxds 
+ \iint_{(0,t)\times \T} \rho \nabla V \phi\,dxds = \\
& - r_0\iint_{(0,t)\times \T}u \phi\,dxds 
- r_1\iint_{(0,t)\times \T} \rho|u_N|^2 u_N \phi\,dxds \\
& + \iint_{(0,t)\times \T} \rrho \nabla^2 \rrho : \nabla \phi\,dxds 
- \iint_{(0,t)\times \T} \nabla \rrho \otimes \nabla \rrho : \nabla \phi\,dxds \\
& + \delta \iint_{(0,t)\times \T} \rho \nabla \Delta^9 \rho \phi\,dxds 
- \iint_{(0,t)\times \T} \rho u_N \phi\,dxds.
\end{aligned}
\end{equation}
The fixed point argument is very similar to the one in \cite{Feireisl} and we omit the proof here. One can show that there exists a smooth solution $(\rho_N, u_N, V_N)$ of system \eqref{2.6} on a short time interval $[0,T(N)]$ for $T(N)< T$ in the space $C([0,T];X_N)$. In order to prove that $T(N)=T$ for any $N$ one can prove that $||u_N||_{L^\infty_t H^2_x} \leq C(N, E_0)$. See also \cite{Vasseur,LV} for further details.
Using Lemma \ref{prop S}--(ii) one deduces that there exists a constant $\zeta = \zeta(N, E_0)$ such that
\begin{equation}\label{rho}
0< \zeta(N,E_0) \leq \rho_N \leq \frac{1}{\zeta(N,E_0)} \text{ for all } t \in (0,T(N)).
\end{equation} 
The energy inequality associated to the approximating sequence is given by
\begin{equation}\label{energy}
\begin{aligned}
& \frac{d}{dt} \int_{\T} E(\rho_N, u_N, V_N)\,dx 
+ \int_{\T} \rho_N |Du_N|^2\,dx 
+ \frac{4 \e}{\gamma} \int_{\T} |\nabla \rho_N^{\gamma/2}|^2\,dx \\
& + \frac{2}{5} \e \eta \int_{\T} |\nabla \rho_N^{-5}|^2\,dx
+ \delta \e \int_{\T}|\Delta^5 \rho_N|^2\,dx
+ r_0 \int_{\T} |u_N|^2\,dx \\
& + r_1 \int_{\T} \rho_N |u_N|^4dx  + \mu \int_{\T} |\Delta u_N|^2\,dx \\
& + \e \int_{\T} \rho_N |\nabla^2 \log \rho_N|^2\,dx
+ \int_{\T} \rho_N u_N^2\,dx = 0
\end{aligned}
\end{equation}
where
\begin{equation*}
\begin{aligned}
    E(\rho_N, u_N, V_N) & = \int_{\T} \frac{1}{2}\rho_N |u_N|^2 + \eta \frac{\rho_N^{-10}}{11}  + \frac{\rho_N^{\gamma}}{\gamma - 1} dx \\
& \int_{\T}+ \delta\frac{|\nabla \Delta^4 \rho_N|^2}{2} + 2 |\nabla \sqrt{\rho_N}|^2 + \frac{1}{2} \rho_N |\nabla V_N|^2 dx.
\end{aligned}
\end{equation*}
and in particular
\begin{equation*}
||\rho_N||_{L^{\infty}((0,T);H^9)} \leq C(E_0(\rho_N,u_N),\delta),
\end{equation*}
which together with $\eqref{rho}$ guarantees that $\rho_N(x,t)$ is a positive smooth function for all $(x,t)$.

It remains to show that there exists a smooth solution for the finite dimensional problem of the Poisson equation defined above, namely
\begin{equation}\label{finpoisson}
-\Delta V_N = \rho_N -g_N \text{ on } \T \times [0,T],
\end{equation}
where
\begin{equation*}
g_N = \sum_{1}^{N} g_i(x) e_i(x) \qquad g_i= \sum_{1}^{N} (g,e_i).
\end{equation*}

It is sufficient to preserve the same conditions  $\eqref{meanVapp}$ and $\eqref{compcondapp}$ for $V_N$ and $g_N$, namely:
\begin{equation*}
\dashint_{\T} V_N(x,t) \,dx = 0.
\end{equation*}
According to the continuity equation $\eqref{2.6}_1$ and the positivity of $\rho_N$ \eqref{rho}, there exists $M_N >0$ such that:
\begin{equation*}
\int_{\T} g_N(x) \,dx = \int_{\T} \rho_N(x,t) \,dx = M_N 
\end{equation*}
where
\begin{equation*}
\int_{\T} \rho_N^0(x) \,dx = M_N. 
\end{equation*}
The first step of the proof is now complete: there exists a unique smooth solution $(\rho_N,u_N,V_N)$ for the finite dimensional approximating scheme of \eqref{2.6} in $(0,T) \times \T$.

\begin{remark}\label{rem:1app}
The following estimate will be useful later (see \cite{Vasseur} and \eqref{qpart}):
\begin{equation}
\begin{aligned}
    & \e^{1/2} ||\sqrt{\rho_N}||_{L^2((0,T),H^2(\T))} + \e^{1/4}||\nabla \rho_N^{1/4}||_{L^4((0,T),L^4(\T))}  \\
    & \leq C\e ||\sqrt{\rho_N} \nabla^2 \log \rho_N||_{L^2((0,T),L^2(\T))}
\end{aligned}
\end{equation}
where $C > 0$ is independent of $N$.
\end{remark}

\begin{lemma}\label{unibound2app}
The following estimates hold for any fixed positive constants $\e, \mu, \delta, \eta$:
\begin{equation}\label{uni bound 2}
\begin{aligned}
& ||\sqrt{\rho_N}||_{L^2((0,T); H^2(\T))} + ||\partial_t(\sqrt{\rho_N})||_{L^2((0,T); L^2(\T))} \leq K, \\
&||\rho_N||_{L^2((0,T); H^{10})} + ||\partial_t(\rho_N)||_{L^2((0,T); L^2(\T))} \leq K, \\
&||\rho_N u_N||_{L^2((0,T); L^2(\T))} + ||\partial_t(\rho_N u_N)||_{L^2((0,T); L^2(\T))} \leq K,\\
&\nabla(\rho_N u_N) \text{ is unif. bound. in } L^4((0,T);L^{6/5}(\T)) + L^2((0,T);L^{3/2}(\T)),\\
&||\rho_N^{\gamma}||_{L^{5/3}((0,T); \T)} \leq K,\\
&||\rho_N^{-10}||_{L^{5/3}((0,T); \T)} \leq K, \\
&||\nabla V_N||_{L^{\infty}((0,T);H^1(T))}\leq K,
\end{aligned}
\end{equation}
where $K=K(\e,\delta,\mu, \eta)$ is independent of $N$. 
\end{lemma}

\begin{proof}
Thanks to Remark \ref{rem:1app} it is clear that
\begin{equation*}
||\sqrt{\rho_N}||_{L^2((0,T); H^2(\T))} \leq C.
\end{equation*}
Moreover,
\begin{equation*}
2 \partial_t(\sqrt{\rho_N}) = - \sqrt{\rho_N} \dive u_N - 2 \nabla \sqrt{\rho_N} u_N = - \sqrt{\rho_N} \dive u_N - 4 \rho_N^{1/4}u_N \nabla \rho_N^{1/4},
\end{equation*}
which gives $\partial_t(\sqrt{\rho_N})$ uniformly bounded in $L^2((0,T); L^2(\T))$.
\begin{equation*}
\partial_t(\rho_N) = - \rho_N\dive u_N - 4(\nabla \rho^{1/4})(\rho_N^{1/4}u_N)\sqrt{\rho_N},
\end{equation*}
thus 
\begin{equation*}
\begin{aligned}
&||\partial_t(\rho_N) ||_{L^2((0,T),L^2(\T))} \leq  ||\rho_N^{1/2} ||_{L^{\infty}((0,T),L^{\infty}(\T))} ||\sqrt{\rho_N} Du_N ||_{L^2((0,T),L^2(\T))} \\
& +4 ||\nabla \rho_N^{1/4} ||_{L^4((0,T),L^4(\T))} ||\rho_N^{1/4} u_N ||_{L^4((0,T),L^4(\T))} ||\rho_N^{1/2} ||_{L^{\infty}((0,T),L^{\infty}(\T))}\leq C
\end{aligned}
\end{equation*}
and $||\rho_N||_{L^2_t H_x^{10}(\T)} \leq C$ comes from $\eqref{energy}$. \\[0.2em]
The uniform boundedness of $\partial_t(\rho_N u_N)$ in $L^{2}((0,T);H^{-9}(\T))$ comes from $ \rho_N \in L^2((0,T);H^{10}(\T))$ and
\begin{equation*}
\begin{aligned}
\partial_t(\rho_N u_N) & = - \dive(\rho_N u_N \otimes u_N) - \nabla \rho_N^{\gamma} + \eta \nabla \rho_N^{-10} \\
& + \mu \Delta^2 u_N + \dive(\rho_N Du_N)  - r_0 u_N - r_1 \rho_N |u_N|^2u_N  \\
& + \e \nabla \rho_N \nabla u_N + \dive (\rho_N \nabla^2 \log \rho_N) + \delta \rho_N \nabla \Delta^9 \rho_N - \rho_N u_N,
\end{aligned}
\end{equation*}
so that
\begin{equation*}
||\rho_N u_N||_{L^2((0,T);L^2(\T))} \leq ||\rho_N^{3/4}||_{L^{\infty}((0,T);L^4(\T))} ||\rho_N^{1/4}u_N||_{L^{4}((0,T);L^4(\T))} \leq C.
\end{equation*}
Moreover,
\begin{equation*}
\nabla (\rho_N u_N)= 4 \sqrt{\rho_N} \rho_N^{1/4} u_N \nabla \rho_N^{1/4} + \sqrt{\rho_N}  \sqrt{\rho_N} \nabla u_N,
\end{equation*}
and since $\sqrt{\rho}_N \in L^{\infty}_t L^3_x$, $\nabla \rho_N^{1/4} \in L^4_t L^4_x$ and $\rho^{1/4}u \in L^4_t L^4_x$, then $\nabla (\rho_N u_N)$ is uniformly bounded in $L^4((0,T);L^{6/5}(\T)) + L^2((0,T);L^{3/2}(\T))$.\\[0.2em]
Thanks to the uniform bound in $\eqref{energy}$, the Sobolev embedding theorem and interpolation inequalities, one gets:
\begin{equation*}
\begin{aligned}
&||\rho_N^{\gamma}||_{L^{5/3}((0,T); \T)} \leq ||\rho_N^{\gamma}||^{2/5}_{L^{\infty}((0,T);L^1(\T))}  ||\rho_N^{\gamma}||^{3/5}_{L^1((0,T);L^3(\T))} \leq K, \\
&||\rho_N^{-10}||_{L^{5/3}((0,T); \T)} \leq ||\rho_N^{-10}||^{2/5}_{L^{\infty}((0,T);L^1(\T))}  ||\rho_N^{-10}||^{3/5}_{L^1((0,T);L^3(\T))} \leq K.
\end{aligned}
\end{equation*}
Finally, from \eqref{energy} we have $||\nabla V_N||_{L^{\infty}((0,T);L^2(\T))} \leq C$. \\ Since $g_N \in C([0,T];L^2(\T))$, $\rho_N \in C([0,T];L^q(\T))$ for any $q<3$, and
\begin{equation*}
    - \Delta V_N= \rho_N - g_N,
\end{equation*}
we obtain $V_N \in L^{\infty}((0,T);H^2(\T))$.
\end{proof}

At this point we are able to apply the Aubin--Lions lemma and obtain the following result.
\begin{lemma}\label{convergence result}
 There exist $\rho \geq 0 \in L^2((0,T);H^{10}(\T))$,\\ $u \in L^2((0,T);H^2(\T))$ and $V\in C([0,T];H^1(\T))$ such that the following hold:
 \begin{equation}\label{rhoN}
\rho_N \rightarrow \rho \text{ strongly in } L^2((0,T);H^9(\T)), \text{ weakly in } L^2((0,T);H^{10}(\T)),
\end{equation}
\begin{equation}
\sqrt{\rho_N} \rightarrow \sqrt{\rho} \text{ strongly in } L^2((0,T);H^1(\T)), \text{ weakly in } L^2((0,T);H^2(\T)),
\end{equation}
\begin{equation}\label{rhou}
\rho_N u_N \rightarrow \rho u \text{ strongly in } L^2((0,T);L^2(\T)),
\end{equation}
\begin{equation}\label{uapp}
u_N \rightarrow u \quad \text{ strongly in }  L^2((0,T);H^1(\T)), \; \text{ weakly in } L^2((0,T);H^2(\T)),
\end{equation}
\begin{equation}\label{papp}
\rho_N^{\gamma} \rightarrow \rho^{\gamma} \text{ strongly in } L^1((0,T);L^1(\T)),
\end{equation}
\begin{equation}\label{inverserho}
\rho_N^{-10} \rightarrow \rho^{-10} \text{ strongly in } L^1((0,T) \times \T),
\end{equation}
and
\begin{equation}\label{Vapp}
    \nabla V_N \rightarrow \nabla V \text{ strongly in } C([0,T];L^2(\T)).
\end{equation}
\end{lemma}

\begin{proof}
Taking into account the uniform bounds \eqref{uni bound 2} in Lemma \ref{unibound2app}, one can apply the Aubin--Lions lemma directly. We only make some comments about \eqref{inverserho}. In order to prove the strong convergence of \\ $\rho_N^{-10}$ in $L^1((0,T);L^1(\T))$ one has to recall the following Sobolev inequality (see \cite{BD})
\begin{equation}\label{sobolev}
||\rho^{-1}||_{L^{\infty}(\T)} \leq C(1+ ||\rho||_{H^{k+2}(\T)})^2(1+ ||\rho^{-1}||_{L^{3}(\T)})^3
\end{equation}
for $k \geq 3/2$. Combining $\eqref{sobolev}$ and $\rho^{-1}_N \in L^{\infty}_t L^{10}_x$ one can deduce
\begin{equation}\label{rhop}
||\rho_N||_{L^{\infty}((0,T) \times \T)} \geq C(\delta,\eta) >0 \text{ a.e. in }(0,T) \times \T.
\end{equation}
From \eqref{rhoN} we know there exists a subsequence such that
\begin{equation*}
\rho_N^{-10} \rightarrow \rho^{-10} \text{ a.e. in } (0,T) \times \T,
\end{equation*}
and thus, using $\eqref{uni bound 2}_5$ and Vitali's theorem, we obtain \eqref{inverserho} directly. We underline also that, using the strong convergence \eqref{rhou} and \eqref{uapp}, we get
\begin{equation*}
\rho_N u_N \otimes u_N \rightarrow \rho u \otimes u \text{ strongly in } L^1((0,T) \times \T).
\end{equation*}
\end{proof}

The following lemma concludes this part.
\begin{lemma}[Lemma 2.3, \cite{Vasseur}]\label{lemma 2.3}
When $N \rightarrow \infty$ we have:
\begin{equation}
\rho_N |u_N|^2 u_N \rightarrow \rho |u|^2 u \text{ strongly in } L^1((0,T)\times \T). 
\end{equation}
\end{lemma}

Finally, we are able to pass to the limit $N \rightarrow \infty$. From $\eqref{rhoN}$ and $\eqref{rhou}$ one gets:
\begin{equation*}
\partial_t(\rho) + \dive(\rho u) = \e \Delta \rho \qquad \text{ a.e. in } (0,T) \times \T.
\end{equation*}
Also the Poisson equation is satisfied pointwise:
\begin{equation*}
-\Delta V = \rho - g \qquad \text{ a.e. in } (0,T) \times \T,
\end{equation*}
thanks to $\eqref{Vapp}$, $\eqref{rhoN}$ and the definition of $g_N$. Moreover, the weak formulation holds for any test function $\phi$:
\begin{equation}
\begin{aligned}
& \int_{\T} \rho u(T)\phi dx 
- \int_{\T} m^0 \phi dx 
+ \mu \iint_{(0,t)\times\T} \Delta u \Delta \phi dxds 
- \iint_{(0,t)\times\T} \rho^\gamma \nabla \phi dxds  \\
& - \iint_{(0,t)\times\T} \rho u \otimes u : \nabla \phi dxds 
+ \iint_{(0,t)\times\T} \rho D(u): \nabla \phi\,dxds \\
& + \eta \iint_{(0,t)\times\T} \rho^{-10} \nabla \phi\,dxds 
+ \e \iint_{(0,t)\times\T} \nabla \rho\, \nabla u\, \phi\,dxds \\
&+ \iint_{(0,t)\times\T} \rho \nabla V \phi\,dxds 
= - r_0 \iint_{(0,t)\times\T} u \phi dxds
- r_1 \iint_{(0,t)\times\T} \rho |u|^2 u \phi dxds \\
& + \iint_{(0,t)\times\T} \sqrt{\rho}\,\nabla^2\sqrt{\rho} : \nabla\phi\,dxds  
- \iint_{(0,t)\times\T} \nabla\sqrt{\rho} \otimes \nabla\sqrt{\rho} : \nabla\phi\,dxds \\
& + \delta \iint_{(0,t)\times\T} \rho \nabla\Delta^9 \rho\, \phi\,dxds
- \iint_{(0,t)\times\T} \rho u \phi\,dxds.
\end{aligned}
\end{equation}
The energy inequality $\eqref{energy}$ becomes $\eqref{limit energy1}$ by the weak lower semicontinuity of the norms:
\begin{equation*}
\begin{aligned}
&\sup_{t \in [0,T]} E(\rho,u,V) + \iint_{(0,t)\times\T} \rho |Du|^2 dx ds + \frac{4 \e}{\gamma} \iint_{(0,t)\times\T} |\nabla \rho^{\gamma/2}|^2 dx ds \\
&+ \frac{2}{5} \e \eta \iint_{(0,t)\times\T}|\nabla \rho^{-5}|^2 dx ds + \delta \e \iint_{(0,t)\times\T}|\Delta^5 \rho|^2 dx ds \\
& + r_0 \iint_{(0,t)\times\T} |u|^2 dx ds + r_1 \iint_{(0,t)\times\T}\rho |u|^4 dx ds + \mu \iint_{(0,t)\times\T} |\Delta u|^2 dx ds \\
&+ \e \iint_{(0,t)\times\T} \rho|\nabla^2 \log \rho|^2 dx ds + \iint_{(0,t)\times\T} \rho|u|^2 dx ds \leq E_0
\end{aligned}
\end{equation*}
where
\begin{equation*}
E_0= \int_{\T} \frac{1}{2}\rho_0|u_0|^2 + \frac{\eta}{11} (\rho_0)^{-10} + \frac{(\rho_0)^{\gamma}}{\gamma - 1} + \delta \frac{|\nabla \Delta^4 \rho_0|^2}{2} + 2 |\nabla \sqrt{\rho_0}|^2 + \frac{1}{2}|\nabla V_0|^2dx.
\end{equation*}

\subsubsection{BD entropy}\label{subsec2}
In order to conclude the proof of Proposition \ref{1 approx} it remains to show the validity of \eqref{bdinit}. In this part we derive the BD entropy associated to system $\eqref{2.6}$, which is equivalent to rewrite the energy inequality in terms of the new state variables $(\rho,w,V)$, where $w = u + \nabla \log \rho$. Indeed, the compactness achieved for $\rho$ in \eqref{rhoN} and \eqref{rhop} makes $\phi = \nabla \log \rho$ an admissible test function to test the weak formulation.

\begin{lemma}
Let $(\rho,u,V)$ be as in Proposition \ref{1 approx}, and let $w= u + \nabla \log \rho$. Then the following estimate holds:
\begin{equation}\label{BD}
\begin{aligned}
&\frac{d}{dt} \int_{\T} \frac{1}{2} \rho |w|^2 + \frac{\rho^{\gamma}}{\gamma - 1} + \frac{1}{2}|\nabla \sqrt{\rho}|^2 + \eta \frac{2}{5} \frac{\rho^{-10}}{11} + [\rho(\log \rho -1)+1] dx \\
&\frac{d}{dt} \int_{\T} - r_0 \log \rho + \delta \frac{|\nabla \Delta^4 \rho|^2}{2} + \frac{1}{2}|\nabla V|^2 dx \\
& + \int_{\T} \rho|A(u)|^2 dx + \frac{4\e}{\gamma} \int_{\T} |\nabla \rho^{\gamma/2}|^2 dx + \frac{4}{\gamma} \int_{\T} |\nabla \rho^{\gamma/2}|^2 dx + r_0 \int_{\T} \rho u^2 dx \\
&+ r_1 \int_{\T} \rho u^4 dx + \int_{\T} \rho(\rho - g) dx + \int_{\T} \rho u^2 dx + \e \int_{\T} \rho |\nabla \log \rho|^2 dx \\
& + \frac{2\e\eta}{5} \int_{\T} |\nabla \rho^{-5}|^2 dx + \frac{2\eta}{5} \int_{\T} |\nabla \rho^{-5}|^2 dx + \mu \int_{\T} |\Delta u|^2 dx \\
&+ \delta\e \int_{\T} |\Delta^5 \rho|^2 dx + \delta \int_{\T} |\Delta^5 \rho|^2 dx + \e \int_{\T} |\Delta \rho|^2/\rho dx + \int_{\T} \rho |\nabla^2 \log \rho|^2 dx \\
&+ \e \int_{\T} \rho |\nabla^2 \log \rho|^2 dx = \sum_{i=1}^{6} R_i \\
&= - \e \int_{\T} \nabla \rho \nabla u \nabla \log \rho dx + \frac{1}{2}\e \int_{\T} \Delta \rho |\nabla \log \rho|^2 dx - r_1 \int_{\T} \rho u^2 u \nabla \log \rho dx \\
&\qquad - \e \int_{\T} \Delta \rho \dive(\rho u)/\rho dx - \mu \int_{\T} \Delta u \nabla \Delta \log \rho dx + r_0 \e \int_{\T} \Delta \rho/\rho dx .
\end{aligned}
\end{equation}	
\end{lemma}

\begin{proof}
Let $w= u + \nabla \log \rho$. Then the equivalent formulation of $\eqref{1 approx}$ is:
\begin{equation}
\begin{aligned}
& \partial_t \rho + \dive(\rho w) = \e \Delta \rho + \Delta \rho, \\
&\partial_t(\rho w) + \dive(\rho w \otimes w) + \nabla \rho^{\gamma} +  \dive(\rho Dw) - \Delta(\rho w) \\
& +  \rho \nabla V - \dive(\rho \nabla^2 \log \rho) = \\
& -r_0u - r_1\rho|u|^2u - \rho u + \eta \nabla \rho^{-10} + \mu \Delta^2 u - \e \nabla \rho \nabla u + \delta \rho \nabla \Delta^9 \rho + \e \nabla \Delta \rho.
\end{aligned}
\end{equation}
Multiplying the momentum equation by $w$, using the new continuity equation and arguing as \cite{BD,BD3, LV} one gets exactly the BD entropy identity \eqref{BD}.
\end{proof}

\begin{remark}
In order to perform the vanishing limits as $\e,\mu$ go to zero, one has to take care of the terms $R_i$. In particular, concerning $R_3$, thanks to Remark \ref{rem:1app} we have
\begin{equation*}
-r_1 \int_{\T} \nabla \rho |u|^2 u \leq C_1 \int_{\T} \rho |u|^4 + C_2 \int_{\T} |\nabla \rho^{1/4}|^4 \leq \tilde{C}E_0.
\end{equation*}
Regarding $R_6$ we deduce
\begin{equation*}
\e r_0\int_{\T} \frac{\Delta \rho}{\rho} \leq C \e ||\rho||_{H^2(\T)} ||\rho^{-1}||_{L^{\infty}(\T)},
\end{equation*}
which is bounded thanks to $\eqref{sobolev}$ and $\eqref{rhoN}$ and goes to zero as $\e \rightarrow 0$. Therefore, the BD entropy \eqref{BD} becomes \eqref{bdinit}. 
\end{remark}

The proof of Proposition \ref{1 approx} is now complete.

\subsection{Vanishing limits as \texorpdfstring{$\e, \mu \rightarrow 0$}{epsilon,mu ->0}}
We now consider the limit $\e, \mu \rightarrow 0$ for the approximate system \eqref{2.6}. Let $(\rho_{\e,\mu}, u_{\e,\mu}, V_{\e,\mu})$ denote the corresponding weak solutions, with $\eta, \delta>0$ fixed. The following holds.

\begin{proposition}\label{prop 3.1}
	There exists a finite energy weak solution $(\rho,u,V)$ of the following system:
	\begin{equation}\label{wsem}
	\begin{aligned}
	&\partial_t \rho + \dive(\rho u)= 0,\\
	&\partial_t(\rho u ) +  \dive(\rho u \otimes u) - \dive(\rho D(u)) + \nabla \rho^{\gamma} - \eta \nabla \rho^{-10} + \rho \nabla V \\
	&\qquad = -r_0 \rho u - r_1 \rho |u|^2 u 
	 + \dive(\rho \nabla^2 \log \rho) - \rho u + \delta \rho \nabla \Delta^9 \rho, \\
	&-\Delta V = \rho - g,
	\end{aligned}
	\end{equation}
	with suitable initial data, for any $T>0$. In particular, weak \\ solutions $(\rho,u,V)$ satisfy the BD-entropy \eqref{BD 3} and the energy inequality \eqref{energy 3}.
\end{proposition}
Let us recall the following uniform bounds.

\begin{lemma}\label{lemma 3.2}
The following estimates hold:
\begin{equation}
\begin{aligned}
&||\partial_t(\sqrt{\rho_{\e,\mu}})||_{L^2((0,T) \times \T)} + ||\sqrt{\rho_{\e,\mu}}||_{L^2((0,T);H^2(\T))} \leq K, \\
&||\partial_t(\rho_{\e,\mu})||_{L^2((0,T);L^{3/2}(\T))} + ||\rho_{\e,\mu}||_{L^{\infty}((0,T);H^{9}(\T))} + ||\rho_{\e,\mu}||_{L^2((0,T);H^{10}(\T))} \leq K\\
&||\partial_t(\rho_{\e,\mu} u_{\e,\mu})||_{L^2((0,T);H^{-9}(\T))} + ||\rho_{\e,\mu}u_{\e,\mu}||_{L^2((0,T) \times \T)} \leq K,\\
&\nabla(\rho_{\e,\mu} u_{\e,\mu}) \text{ is unif. bound. in } L^4((0,T);L^{6/5}(\T)) + L^2((0,T);L^{3/2}(\T)),\\
&||\rho_{\e,\mu}^{\gamma}||_{L^{5/3}((0,T) \times \T)} \leq K, \\
&||\rho_{\e,\mu}^{-10}||_{L^{5/3}((0,T) \times \T)} \leq K, \\
&|| V_{\e,\mu}||_{L^{\infty}((0,T);H^2(\T))} \leq K,
\end{aligned}
\end{equation}
where $K$ is independent of $\e,\mu$.
\end{lemma}

\begin{proof}
To prove Lemma \ref{lemma 3.2} it is enough to follow the same argument as in Lemma \ref{lemma 2.3}, taking into account the energy inequality \eqref{limit energy1} and the BD-entropy \eqref{bdinit}.
\end{proof}

We now apply the classical Aubin–Lions lemma, using the regularity provided by Lemma \ref{lemma 3.2}. There exist $(\rho,u,V)$ such that
\begin{equation}\label{conv}
\begin{aligned}
&\rho_{\e,\mu} \rightarrow \rho \text{ strongly in } C((0,T);H^9(\T)) \text{ and weakly in } L^2((0,T); H^{10}(\T)), \\
&\sqrt{\rho_{\e,\mu}} \rightarrow \sqrt{\rho} \text{ strongly in } L^2((0,T);H^1(\T)) \text{ and weakly } L^2((0,T); H^{2}(\T)),\\
&\rho_{\e,\mu} u_{\e,\mu} \rightarrow \rho u \text{ strongly in } L^2((0,T) \times \T),\\
&u_{\e,\mu} \weakto  u \text{ weakly in} L^2((0,T);H^2(\T)), \\
& \nabla V_{\e,\mu} \rightarrow \nabla V \text{ strongly in } C([0,T];L^2(\T)).
\end{aligned}
\end{equation}
Thanks to Lemma \ref{lemma 3.2}, \eqref{conv}$_1$ and Vitali's theorem, we conclude
\begin{equation}\label{conv2}
\begin{aligned}
&\rho_{\e,\mu}^{\gamma} \rightarrow \rho^{\gamma} \text{ strongly in } L^1((0,T) \times \T), \\
&\rho_{\e,\mu}^{-10} \rightarrow \rho^{-10} \text{ strongly in } L^1((0,T) \times \T).
\end{aligned}
\end{equation}
From \eqref{conv}$_1$ and \eqref{conv2} we infer
\begin{equation}
\begin{aligned}
&\e \nabla \rho_{\e,\mu} \rightarrow 0 \text{ strongly in } L^2((0,T) \times \T),\\
&\e \nabla \rho_{\e,\mu} \nabla u_{\e,\mu} \rightarrow 0 \text{ strongly in } L^1((0,T) \times \T),
\end{aligned}
\end{equation}
and for the viscous term
\begin{equation*}
\left| \mu \iint_{[0,T] \times \T} \Delta^2 u_{\e,\mu} \phi \right| \leq \sqrt{\mu} ||\sqrt{\mu} \Delta u_{\e,\mu}||_{L^2((0,T)\times \T)} || \Delta \phi||_{L^2((0,T)\times \T)} \rightarrow 0 
\end{equation*}
for any $\phi \in L^2((0,T);H^2(\T))$ as $\mu \rightarrow 0$. Thanks to \eqref{conv}$_3$, \eqref{conv}$_4$ and arguing as in Lemma \ref{lemma 2.3} one can show
\begin{equation}\label{conv3}
\begin{aligned}
& \rho_{\e,\mu} u_{\e,\mu} \otimes u_{\e,\mu}  \rightarrow \rho u \otimes u \text{ in the sense of distributions},\\
&\rho_{\e,\mu} |u_{\e,\mu}|^2 u_{\e,\mu} \rightarrow \rho |u|^2 u \text{ strongly in } L^1((0,T);L^1(\T)).
\end{aligned}
\end{equation}

\begin{remark}
At this stage we loose the $H^2$-regularity for $u$. We only retain the boundedness $\sqrt{\rho}\nabla u \in L^2_tL^2_x$ and the positivity of the particle density. The tensor $\MT$ in Definition \ref{def1app} and the dissipation identity \eqref{enedissapp} are obtained as weak limits. In particular,
\begin{equation*}
    \sqrt{\rho}_{\e,\mu} \nabla u_{\e,\mu} \weakto \MT \text{ \; in \;} L^2((0,T) \times \T), \text{\; as \;} \e,\mu \rightarrow 0.
\end{equation*}
On the other hand, since
\begin{equation*}
    \rho_{\e,\mu} \nabla u_{\e,\mu} = \nabla (\rho_{\e,\mu} u_{\e,\mu}) - 2 \sqrt{\rho}_{\e,\mu} u_{\e,\mu} \otimes \nabla \sqrt{\rho}_{\e,\mu},
\end{equation*}
thanks to \eqref{conv} the latter converges for $\mu \rightarrow 0$ to 
\begin{equation*}
\sqrt{\rho} \MT = \nabla(\rho u) -   2 \sqrt{\rho} u\otimes \nabla \sqrt{\rho}   \text{ \; in \;} \mathcal{D}'((0,T)\times\T). 
\end{equation*}
Thus, the tensor $\MT$ is uniquely identified by $(\rho, u)$. Therefore there exist two tensors (depending on $\delta,\eta$) $\MT^s,\MT^a$ such that
\begin{equation*}
    \begin{aligned}
    &\sqrt{\rho_{\e,\mu}}D(u_{\e,\mu}) \weakto \MT^s \text{ in }L^2((0,T);L^2(\T)),\\
    &\sqrt{\rho_{\e,\mu}}A(u_{\e,\mu}) \weakto \MT^a \text{ in }L^2((0,T);L^2(\T)).
    \end{aligned}
\end{equation*}
\end{remark}

\begin{lemma}[Lemma 3.3,  \cite{Vasseur}]\label{lemma 3.3}
If $\rho_{\e,\mu}^{1/4} u_{\e,\mu}$ is bounded in $L^4((0,T);L^4(\T))$, $\rho_{\e,\mu}$ converges a.e. to $\rho$, and $\rho_{\e,\mu} u_{\e,\mu}$ converges a.e. to $\rho u$, then
\begin{equation}
\sqrt{\rho_{\e,\mu}} u_{\e,\mu} \rightarrow \sqrt{\rho}u \text{ strongly in }L^2((0,T);L^2(\T)).
\end{equation}
\end{lemma}

\begin{remark}\label{poisson mu e}
The Poisson equation is satisfied pointwise:
\begin{equation*}
-\Delta V = \rho - g  \text{ holds for a.e. } (t,x) \in [0,T] \times \T,
\end{equation*}
thanks to \eqref{conv}$_1$ and \eqref{conv}$_5$.
\end{remark}

The previous convergence results show that, as $\e, \mu \rightarrow 0$, the weak solution $(\rho_{\e,\mu},u_{\e,\mu}, V_{\e,\mu})$ satisfies the weak formulation of the momentum equation \eqref{wsem}$_2$. The regularity of $\rho$ is enough to conclude that the continuity equation \eqref{wsem}$_1$ and the Poisson equation \eqref{wsem}$_3$ are satisfied pointwise.

By weak lower semicontinuity, the energy inequality \eqref{limit energy1} passes to the limit and becomes 
\begin{equation}\label{energy 3}
\begin{aligned}
&\int_{\T} \frac{1}{2} \rho |u|^2 + \frac{\rho^{\gamma}}{\gamma -1} + \frac{\eta}{11} \rho^{-10} + 2 |\nabla\sqrt{\rho}|^2 + \frac{\delta}{2} |\nabla \Delta^4 \rho|^2 + \frac{1}{2} |\nabla V|^2 dx\\
&+ \iint_{(0,t)\times\T} | \MT^s |^2 dx ds + r_0 \iint_{(0,t)\times\T}|u|^2 dx ds + r_1\iint_{(0,t)\times\T}\rho |u|^4 dx ds \\
& + \iint_{(0,t)\times\T}\rho |u|^2 dx ds \\
&\leq \int_{\T} \frac{1}{2} \rho_0 |u_0|^2 + \frac{\rho_0^{\gamma}}{\gamma -1} + \frac{\eta}{11} \rho_0^{-10} + 2 |\nabla\sqrt{\rho_0}|^2 + \frac{\delta}{2} |\nabla \Delta^4 \rho_0|^2 dx,
\end{aligned}
\end{equation}
and, letting $\e,\mu \to 0$ in the BD-entropy \eqref{bdinit}, we obtain
\begin{equation}\label{BD 3}
\begin{aligned}
&\int_{\T} \frac{1}{2} \rho |w|^2 + \frac{\rho^{\gamma}}{\gamma -1} + \frac{\eta}{11} \rho^{-10} + 2 |\nabla\sqrt{\rho}|^2 + \frac{\delta}{2} |\nabla \Delta^4 \rho|^2 - r_0 \log \rho \\
& + \int_{\T} (\rho(\log \rho -1)+1) + \frac{1}{2}|\nabla V|^2 dx + \iint_{(0,t)\times\T} |\MT^a|^2 dx ds  \\
&+ \frac{4}{\gamma} \iint_{(0,t)\times\T} |\nabla \rho^{\gamma/2}|^2 dx ds + r_0 \iint_{(0,t)\times\T}|u|^2 dx ds + \iint_{(0,t)\times\T} \rho(\rho - g) dx ds\\
&+ r_1 \iint_{(0,t)\times\T}\rho |u|^4 dx ds + \iint_{(0,t)\times\T}\rho |u|^2 dx ds + \frac{2\eta}{5} \iint_{(0,t)\times\T} |\nabla \rho^{-5}|^2 dx ds\\
&+ \delta \iint_{(0,t)\times\T} |\Delta^5 \rho|^2 dx ds + \iint_{(0,t)\times\T} \rho |\nabla^2 \log \rho|^2 dx ds \\
&\leq \int_{\T} \frac{1}{2} \rho_0 |w_0|^2 + \frac{\rho_0^{\gamma}}{\gamma -1} + \frac{\eta}{11} \rho_0^{-10} + 2 |\nabla\sqrt{\rho_0}|^2 + \frac{\delta}{2} |\nabla \Delta^4 \rho_0|^2 - r_0 \log \rho_0 dx\\
&+ \int_{\T} (\rho_0(\log \rho_0 -1)+1) dx + 2E_0.
\end{aligned}
\end{equation}
This concludes the proof of Proposition \ref{prop 3.1}.

\subsection{Vanishing limits as \texorpdfstring{$\delta, \eta \rightarrow 0$}{delta,eta ->0}}
To complete the proof of Theorem \ref{thm1app}, it remains to control the terms $\eta \rho^{-10}$ and $\rho_{\delta} \nabla \Delta^9 \rho_{\delta}$ in the weak formulation of the momentum equation as $\eta,\delta \rightarrow 0$. First, we perform the limit $\eta \rightarrow 0$ for fixed $\delta$. At this stage, one can deduce the same compactness properties \eqref{conv}, \eqref{conv2}, \eqref{conv3} for the solution $(\rho_{\eta},u_{\eta}, V_{\eta})$. 

Let $(\rho_{\eta, \delta},u_{\eta,\delta}, V_{\eta,\delta})$ be a weak solution as in Proposition \ref{prop 3.1}. The next lemma allows us to control $\eta \rho_{\eta}^{-10}$. 

\begin{lemma}[Lemma 3.6, \cite{Vasseur}]\label{lemma 3.6}
	For any $\rho_{\eta}$ defined in Proposition \ref{prop 3.1} we have:
	\begin{equation}
	\eta \iint_{(0,t) \times \T} \rho_{\eta}^{-10} \rightarrow 0
	\end{equation}
	as $\eta \rightarrow 0$.
\end{lemma}

This result implies 
\begin{equation*}
\begin{aligned}
& || \eta \rho_{\eta}^{-10}(t)||_{L^1(\T)} \rightarrow 0 \quad \text{ for a.e. } t \in (0,T), \\
& \sqrt{\eta}\nabla \rho^{-5}_{\eta} \weakto 0 \text{ in } L^2((0,T) \times \T).
\end{aligned}
\end{equation*}

\begin{remark}
At this stage we lose the strict positivity of $\rho$, and therefore the tensor $\mathcal{S}$ appears in the weak formulation. It arises as a weak limit of $\sqrt{\rho}_{\delta,\eta} \nabla^2 \log \rho_{\delta,\eta}$, namely
\begin{equation*}
    \sqrt{\rho}_{\delta,\eta} \nabla^2 \log \rho_{\delta,\eta} \weakto \mathcal{S} \text{ \; in \;} L^2((0,T) \times \T), \text{\; as \;} \eta \rightarrow 0.
\end{equation*}
Therefore, equation \eqref{defsapp} is a consequence of the passage into the limit in the identity
\begin{equation*}
    \rho_{\delta,\eta} \nabla^2 \log \rho_{\delta,\eta} = 2 \sqrt{\rho}_{\delta,\eta} \nabla^2 \sqrt{\rho}_{\delta,\eta} - 2 \nabla \sqrt{\rho}_{\delta,\eta} \otimes \nabla \sqrt{\rho}_{\delta,\eta}.
\end{equation*}
Indeed, as $\eta \rightarrow 0$, we obtain
\begin{equation*}
    \sqrt{\rho} \mathcal{S} = 2 \sqrt{\rho} \nabla^2 \sqrt{\rho} - 2 \nabla \sqrt{\rho} \otimes \nabla \sqrt{\rho}  \text{ \; in \;} \mathcal{D}'((0,T)\times\T).
\end{equation*}
\end{remark}

It remains to take the limit $\delta \rightarrow 0$ and control the term $\rho_{\delta} \nabla \Delta^9 \rho_{\delta}$. The following result holds.

\begin{lemma}[Lemma 3.7, \cite{Vasseur}]
	For any $\rho_{\delta}$ defined as in Proposition \ref{prop 3.1}$,$ we have, for any test function $\phi$,
	\begin{equation*}
	\delta \iint_{(0,t)\times\T} \rho_{\delta} \nabla \Delta^9 \rho_{\delta} \phi \rightarrow 0 \text{ as } \delta \rightarrow 0.
	\end{equation*}
\end{lemma}

From \eqref{energy 3} and \eqref{BD 3} we know that
\begin{equation*}
\begin{aligned}
&||\sqrt{\delta} \rho||_{L^{\infty}((0,T);H^9(\T))} \leq C, \\
&||\sqrt{\delta} \nabla \rho||_{L^{2}((0,T);H^{10}(\T))} \leq C,
\end{aligned}
\end{equation*}
where $C$ is independent of $\delta$. Since $H^{10} \subset\subset H^9$, we can extract a convergent subsequence in the strong topology of $H^9(\T)$. This implies 
\begin{equation*}
\begin{aligned}
&|| \sqrt{\delta } \nabla \Delta^4 \rho_{\delta} (t)||_{L^2(\T)} \rightarrow 0 \quad \text{ for a.e. } t \in [0,T],\\
& \sqrt{\delta } \nabla \Delta^5 \rho_{\delta} \weakto 0 \text{ weakly in } L^2((0,T)\times \T).
\end{aligned}
\end{equation*}

By weak lower semicontinuity, the energy inequality \eqref{energy 3} becomes
\begin{equation*}
\begin{aligned}
&\int_{\T} \frac{1}{2} \rho |u|^2 + \frac{\rho^{\gamma}}{\gamma -1} + 2 |\nabla\sqrt{\rho}|^2 + \frac{1}{2} |\nabla V|^2 dx + \iint_{(0,t)\times\T}|\mathcal{T}^s|^2 dx ds \\
&+ r_0 \iint_{(0,t)\times\T}|u|^2 dx ds + r_1\iint_{(0,t)\times\T}\rho |u|^4 dx ds + \iint_{(0,t)\times\T}\rho |u|^2 dx ds \\
&\leq \int_{\T} \frac{1}{2} \rho_0 |u_0|^2 + \frac{\rho_0^{\gamma}}{\gamma -1} + 2 |\nabla\sqrt{\rho_0}|^2 dx,
\end{aligned}
\end{equation*}
and the BD-entropy \eqref{BD 3} yields
\begin{equation*}
\begin{aligned}
&\int_{\T} \frac{1}{2} \rho |w|^2 + \frac{\rho^{\gamma}}{\gamma -1} + 2 |\nabla\sqrt{\rho}|^2 - r_0 \log \rho + (\rho(\log \rho -1)+1) + \frac{1}{2}|\nabla V|^2 dx\\
&+ \iint_{(0,t)\times\T} |\MT^a|^2 dx ds + \frac{4}{\gamma} \iint_{(0,t)\times\T} |\nabla \rho^{\gamma/2}|^2 dx ds + r_0\iint_{(0,t)\times\T}|u|^2 dx ds \\
& + \iint_{(0,t)\times\T} \rho(\rho - g) dx ds + r_1\iint_{(0,t)\times\T}\rho |u|^4 dx ds + \iint_{(0,t)\times\T}\rho |u|^2 dx ds \\
& + \iint_{(0,t)\times\T} |\mathcal{S}|^2 dx ds \\
&\leq \int_{\T} \frac{1}{2} \rho^0 |w_0|^2 + \frac{(\rho_0)^{\gamma}}{\gamma -1} + 2 |\nabla\sqrt{\rho_0}|^2 + (\rho_0(\log \rho_0 -1)+1) dx + 2E_0.
\end{aligned}
\end{equation*}

We conclude that the limit $(\rho,u, V)$ solves \eqref{sys1app} in the sense of Definition \ref{def1app}. The proof of Theorem \ref{thm1app} is now complete.

\section{Existence of global weak solutions to Quantum Navier-Stokes Poisson system in presence of high friction}\label{sec2}
In this section we finally prove Theorem \ref{teo:3}. We will establish the global existence of finite-energy weak solutions for large initial data to system~\eqref{1.1}. We recall that for a given function 
$g : \T \to \R$, it is defined on $(0,T)\times \T$:
\begin{equation}\label{1.1}
\left\{
\begin{aligned} 
\partial_t \rho + \dive(\rho u ) &= 0, \\ 
\partial_t(\rho u) + \dive(\rho u \otimes u) - \dive (\rho Du) + \nabla \rho^{\gamma} + \rho \nabla V 
&= \dive(\rho \nabla^2 \log \rho) - \rho u, \\
-\Delta V &= \rho - g ,
\end{aligned}
\right.
\end{equation}
together with the initial and compatibility conditions
\begin{equation}\label{1.3}
\begin{aligned}
& \rho(0,x)=\rho^0(x), \qquad (\rho u)(0,x)= \rho^0(x)u^0(x),\\[2mm]
& \dashint_{\T} V(x,t)\, dx = 0, \qquad \int_{\T} \rho^0(x)\, dx = \int_{\T} g(x)\, dx .
\end{aligned}
\end{equation}
 As already mentioned, in order to overcome the degeneracy of the viscosity near vacuum regions, we construct approximate solutions to the following modified system:
\begin{equation}\label{1.8}
\left\{
\begin{aligned} 
& \partial_t \rho_\e + \dive(\rho_\e u_\e ) = 0, \\ 
& \partial_t(\rho_\e u_\e) + \dive(\rho_\e u_\e \otimes u_\e) - \dive (\rho_\e Du_\e) 
    + \nabla \rho_\e^{\gamma}  \\
& \qquad\qquad\qquad\qquad = \dive(\rho_\e \nabla^2 \log \rho_\e) 
    - \rho_\e u_\e - \e u_\e - \e \rho_\e |u_\e|^2 u_\e ,\\
& -\Delta V_{\e} = \rho_\e - g_{\e} .
\end{aligned}
\right.
\end{equation}

To this end, we introduce the classical notion of \emph{renormalized weak solutions} for the 
Quantum Navier--Stokes--Poisson system. Owing to the weak integrability of the velocity field and the possible presence of vacuum, the classical weak formulation of the momentum equation can not be applied directly. As shown by Lacroix-Violet and Vasseur~\cite{LV}, velocity renormalization allows one to reformulate the momentum equation by testing it against truncated functions $b(u)$ rather than against $u$ itself. This approach prevents the loss of compactness in the nonlinear terms and ensures the stability of the limiting procedure applied to the approximating solutions introduced in the previous section.

Let us observe that system~\eqref{1.8} is a particular case of system~\eqref{sys1app}, with $r_0=r_1=\e$. Therefore, the existence of weak solutions for~\eqref{1.8} has already been established in Section~\ref{sec:1}. The renormalized formulation is particularly effective in treating nonlinear quantities such as
\[
\rho u, \qquad \rho |u|^{2}u, \qquad \rho Du ,
\]
which do not possess sufficient regularity to converge under weak convergence alone.

Starting from a suitable truncated formulation of the momentum equation, one proves that, as $\e \to 0$, the finite-energy weak solutions of~\eqref{1.8} converge to finite-energy weak solutions of the limit system~\eqref{1.1}.

\subsection{Definition of weak solutions}
We start by giving the definition of weak solution for $\eqref{1.1}$.
\begin{definition}[\textit{Weak Solutions}]\label{def:1}
A triple $(\rho,u, V)$ with $\rho \geq 0$ is said to be a weak solution of the Cauchy problem $\eqref{1.1}$ and $\eqref{1.3}$ if the following conditions are satisfied:
\begin{itemize}
\item Integrability conditions:
\begin{equation*}
\begin{aligned}
& \sqrt{\rho} \in L^{\infty}((0,T);H^1(\T)) \cap L^2((0,T);H^2(\T)) \cap C(0,T;L^2(\T))\\
& \sqrt{\rho} u \in L^{\infty}((0,T);L^2(\T)) \cap L^2((0,T);L^2(\T))  \\
&V \in C(0,T;H^{2}(\T)) \\
&\rho^{\gamma/2} \in L^{\infty}((0,T);L^2(\T)) \cap L^2((0,T);H^1(\T)) \\
& \nabla \sqrt{\rho} \in L^{\infty}((0,T);L^2(\T))\\
&\MT \in L^2((0,T);L^2(\T)) \\
&\rho u \in C([0,T]; L^{3/2}_{weak}(\T))
\end{aligned}
\end{equation*}
\item Continuity equation: \\
For any $\phi \in C_c^{\infty}((0,T) \times \T; \mathbb{R})$:
\begin{equation}\label{eq: cont}
\int_{\T} \rho^0 \phi(0) + \iint_{(0,T) \times \T} \rho \phi_t + \sqrt{\rho} \sqrt{\rho} u \cdot \nabla \phi = 0
\end{equation}
\item  Momentum equation:\\
For any $\psi \in C_c^{\infty}((0,T) \times \T; \mathbb{R}^3)$:
\begin{equation}\label{eq: mom}
\begin{aligned}
&\int_{\T} \rho^0 u^0 \phi(0) dx + \iint_{(0,t)\times\T} \sqrt{\rho} \sqrt{\rho} u \psi_t dx ds \\
& + \iint_{(0,t)\times\T} \sqrt{\rho} u \otimes \sqrt{\rho} u : \nabla \psi dx ds - \iint_{(0,t)\times\T} \sqrt{\rho} \MT^{s} : \nabla \psi dx ds\\
& - \iint_{(0,t)\times\T} \nabla \rho^{\gamma} \psi dx ds + \iint_{(0,t)\times\T} \rho u \psi dx ds \\
&+ \iint_{(0,t)\times\T} \rho \nabla V \psi dx ds - \iint_{(0,t)\times\T} \nabla \rrho \otimes \nabla \rrho : \nabla \psi dx ds \\
& + \iint_{(0,t)\times\T} \rrho \nabla^2 \rrho : \nabla \psi dx ds = 0
\end{aligned}
\end{equation}
\item Poisson Equation:
\begin{equation}\label{poi2}
-\Delta V = \rho - g \qquad \text{ a.e in } (0,T) \times \T.
\end{equation}
\item Energy Dissipation: \\
For any $\Phi \in \in C_c^{\infty}((0,T) \times \T; \mathbb{R}) $
\begin{equation}\label{eq: diss}
\begin{aligned}
  \iint_{(0,t)\times\T} \sqrt{\rho} \MT \Phi dx ds =& - \iint_{(0,t)\times\T} \rho u \nabla \Phi dx ds \\& - 2 \iint_{(0,t)\times\T} \sqrt{\rho} u \otimes \nabla \sqrt{\rho} \Phi dx ds  
\end{aligned}
\end{equation}
\item Energy inequality: \\
There exists $\Lambda \in L^{\infty}((0,T);L^2(\T))$ with $\rho u = \sqrt{\rho}\Lambda$ a.e in $(0,T) \times \T$ such that:
\begin{equation}\label{eq: ene ineq}
\begin{aligned}
& \sup_{t \in (0,T)} \int_{\T} \frac{|\Lambda|^2}{2} + \frac{\rho^{\gamma}}{\gamma -1} + 2 |\nabla \sqrt{\rho}|^2 dx
+ \iint_{(0,t)\times\T} |\MT^s|^2 dx ds \\
& + \iint_{(0,t)\times\T} \rho u^2 dx ds \leq \int_{\T} \frac{\rho^0 (u^0)^2}{2} + \frac{(\rho^0)^{\gamma}}{\gamma -1} + 2 |\nabla \sqrt{\rho^0}|^2 dx
\end{aligned}
\end{equation}
\item BD-Entropy: \\
 Let $w=u+\nabla\log\rho$, there exists a tensor $\mathcal{S} \in L^{2}((0,T);L^{2}(\T))$ such that for a.e. $t\in(0,T)$:
\begin{equation}\label{eq:dispdissapp}
\sqrt{\rho} \mathcal{S}=2 \sqrt{\rho}\nabla^2\sqrt{\rho}-2\nabla\sqrt{\rho}\otimes\nabla\sqrt{\rho}\,\mbox{ a.e. in }(0,T)\times\T,	
\end{equation}		
and: 
\begin{equation}\label{eq: bd3}
\begin{aligned}
&\int_{\T} \frac{1}{2} \rho |w|^2 + \frac{\rho^{\gamma}}{\gamma -1} + 2 |\nabla\sqrt{\rho}|^2 + (\rho(\log \rho -1)+1) + \frac{1}{2}|\nabla V|^2 \left(t\right) dx\\
&+ \iint_{(0,t)\times\T} |\MT^a|^2 dx ds + \frac{4}{\gamma} \iint_{(0,t)\times\T} |\nabla \rho^{\gamma/2}|^2 dx ds \\
&+ \iint_{(0,t)\times\T} \rho(\rho - g) dx ds+ \iint_{(0,t)\times\T}\rho |u|^2 dx ds + \iint_{(0,t)\times\T} |S|^2 dx ds \\
&\leq \int_{\T} \frac{1}{2} \rho^0 |w^0|^2 + \frac{(\rho^0)^{\gamma}}{\gamma -1} + 2 |\nabla\sqrt{\rho^0}|^2 + (\rho^0(\log \rho^0 -1)+1) dx + 2E_0.
\end{aligned}
\end{equation}
\item Initial data: \\
\begin{equation}\label{assumption}
\begin{aligned}
&\rho^0 \geq 0 , \; \; \rho^0 \in L^1 \cap L^{\gamma}(\T) \; \; \nabla \sqrt{\rho^0} \in L^2(\T), \; \; \log \rho^0 \in L^1(\T)\\
& \sqrt{\rho^0}u^0 \in L^2(\T) \; \; \rho^0 u^0 \in L^{3/2}(\T) 
\end{aligned}
\end{equation}
\end{itemize}
\end{definition}
\begin{remark}
We recall that the presence of the tensor $\MT$ in Definition $\ref{def:1}$ is due to the possible presence of vacuum regions. Indeed, if the density is bounded away from zero, \eqref{eq: diss} implies that $\MT=\sqrt{\rho}\nabla u$. However, also in the case when the vacuum has zero Lebesgue measure, if no damping terms appear in the momentum equation, $u$ can not be defined as measurable vector field, as well as $\nabla u$ can not be defined in a distributional sense. Infact, the tensor $\MT$ arises as a weak $L^{2}$-limit of the sequence $\{\sqrt{\re}\nabla\ue)\}_{\e}$, where $\{\re,\ue\}_{\e}$ is a suitable sequence of approximations as in Definition \ref{def1app} of Section \ref{sec:1}. 
The subscripts $\e$ refer to the particular choice $r_0=r_1=\e$ in the approximating procedure. The same happens for the tensor $\mathcal{S}$, it arises as an $L^2$ weak limit of the approximating sequence $\{\sqrt{\re}\nabla^2\log\re\}_{\e}$ present in the BD-entropy $\eqref{bd1}$ of Section \ref{sec:1}.
\end{remark}

The main Theorem is the following:

\begin{theorem}\label{teo:3}
Assume $\rho^0$, $\rho^0 u^0$, $V$ and $g$ as in $\eqref{1.3}$ satisfying $\eqref{assumption}$. Then there exists a finite energy weak solution $(\rho, u, V)$ of $\eqref{1.1}$ in the sense of Definition \ref{def:1}.
\end{theorem}

\subsection{Definition of Renormalized Weak Solutions}
Let us start by presenting the truncation functions (see \cite{AS1}) and their properties involved in the regularization procedure. Let $\bar{\beta}: \mathbb{R} \rightarrow \mathbb{R}$ be an even positive compactly supported smooth function such that:
\begin{equation*}
\bar{\beta}(z) = 1 \text{ for } z \in [-1,1],
\end{equation*}
$supp \bar{\beta} \subset (-2,2)$ and $0 \leq \bar{\beta} \leq 1$. Given $\bar{\beta}$, the function $\tilde{\beta}: \mathbb{R} \rightarrow \mathbb{R}$ is defined as follows:
\begin{equation*}
\tilde{\beta}(z) = \int_{0}^{z} \bar{\beta}(s)ds.
\end{equation*} 
Moreovore for $y \in \mathbb{R}^3$ and for any $\delta >0$ defines:
\begin{equation*}
\begin{aligned}
& \beta_{\delta}^1(y) = \frac{1}{\delta} \tilde{\beta}(\delta y_1) \bar{\beta}(\delta y_2) \bar{\beta}(\delta y_3) \\
& \beta_{\delta}^2(y) = \frac{1}{\delta} \bar{\beta}(\delta y_1) \tilde{\beta}(\delta y_2) \bar{\beta}(\delta y_3) \\
& \beta_{\delta}^3(y) = \frac{1}{\delta} \bar{\beta}(\delta y_1) \bar{\beta}(\delta y_2) \tilde{\beta}(\delta y_3), \\
\end{aligned}
\end{equation*}
for any $l = 1,2,3$ the function $\beta_\delta^l : \mathbb{R}^3 \rightarrow \mathbb{R}$ is a truncation of the function $f(y)= y_l$. For any $\delta >0$ define $\hat{\beta}: \mathbb{R}^3 \rightarrow \mathbb{R}$ as
\begin{equation*}
\hat{\beta}(y) : = \bar{\beta}(\delta y_1) \bar{\beta}(\delta y_2) \bar{\beta}(\delta y_3).
\end{equation*}
The following Lemma collects the main properties of $\beta_\delta^l, \hat{\beta_\delta}$.

\begin{lemma}[Lemma 2.4, \cite{AS1}]\label{lemma: prope}
Let $\delta > 0$ and $K= ||\bar{\beta}||_{W^{2,\infty}}$. Then there exists $C=C(K)$ such that:
\begin{itemize}
	\item For any $\delta >0$ and $l=1,2,3$
	\begin{equation}
	||\beta_{\delta}^l||_{\infty} \leq \frac{C}{\delta} \; \; ||\nabla \beta_{\delta}^l||_{\infty} \leq C \; \; ||\nabla^2 \beta_{\delta}^l||_{\infty} \leq C \delta
	\end{equation}
	\item  For any $\delta >0$
	\begin{equation}
	||\hat{\beta_\delta}||_{\infty} \leq 1 \; \; ||\nabla \hat{\beta_\delta}||_{\infty} \leq C\delta \; \; |\hat{\beta_\delta}| |y| \leq \frac{C}{\delta}
	\end{equation}
	\item The following convergence hold for $l=1,2,3$, pointwise on $\mathbb{R}^3$, as $\delta \rightarrow 0$
	\begin{equation}
	\beta_{\delta}^l(y) \rightarrow y_l \; \; \nabla(\beta_{\delta}^l)(y) \rightarrow \nabla_{y_l} y \; \; \hat{\beta_\delta}(y) \rightarrow 1 
	\end{equation}
\end{itemize}
\end{lemma}

\subsubsection{Di-Perna Lion commutator estimate,\texorpdfstring{ \cite{DPL}}{di perna lions}}
Before going into the regularization procedure, some preliminary results are needed, in particular we present the commutators for convolutions of Di Perna-Lions type \cite{DPL}. For any function $f$ defines $\bar{f}_r$ as the time-space convolution of $f$ with a smooth sequence of even mollifiers $\lbrace \Psi_r \rbrace_r$, namely:
\begin{equation*}
 \bar{f}_r = \Psi_r * f(t,x), \quad t > r
\end{equation*}
where
\begin{equation*}
\Psi_r(t,x) = \frac{1}{r^4} \Psi\left( \frac{t}{r}, \frac{x}{r} \right)
\end{equation*}
and $\Psi$ is a smooth nonnegative even function such that $supp\Psi \subset B_1(0)$ and
\begin{equation*}
\int \int \Psi dxdt = 1.
\end{equation*}

At this point it is possible to construct weak solutions of $\eqref{1.1}$, $\eqref{1.3}$ in the sense of Definition \ref{def:1} as a limit solution of $\eqref{1.8}$ in the sense of Definition $\ref{def1app}$ with $r_0,r_1= \e$ with initial data:
\begin{equation}\label{3.2}
\begin{aligned}
& \re(0,x) = \rho^0(x) \\
& \re \ue(0,x)= \rho^0(x)u^0(x)
\end{aligned}
\end{equation}
satisfying $\eqref{assumption}$, $V_\e$ satisfying $\eqref{meanVapp}$ and $g_\e$ $\eqref{compcondapp}$ such that:
\begin{equation*}
    g_\e(x) \weakto g(x) \text{ in } C([0,T];L^2(\T)).
\end{equation*}

The content of Theorem \ref{teo:3.2} is to present the truncated formulation of the momentum equation, that is equivalent to derive the momentum equation satisfied in the sense of distribution by the \emph{renormalized solutions}. For the sake of completness we recall the definition of the latters and we refer to \cite{LV} for a better exposition. We explicit the QNS case $\eqref{qnsAPP}$ with $\nu=1$.

\begin{definition}[\textit{Renormalized Weak Solutions}]
The couple $(\sqrt{\rho}, \sqrt{\rho} u)$ is a weak renormalized solution to $\eqref{qnsAPP}$ if it verifies the a priori estimates coming from $\eqref{eq: ene ineq}$, $\eqref{eq: bd3}$ and for any function $\phi \in W^{2,\infty}(\R^3)$ there exist two measures $R_{\phi}$, $\bar{R}_{\phi} \in \mathcal{M}(\R^+ \times \Omega)$ such that:
\begin{equation*}
||R_{\phi}||_{\mathcal{M}(\R^+ \times \Omega)} + ||\bar{R}_{\phi}||_{\mathcal{M}(\R^+ \times \Omega)} \leq C||\phi''||_{L^{\infty}(\R)},
\end{equation*}
where the constant $C$ depends only on the solution $(\sqrt{\rho}, \sqrt{\rho} u)$ and for any $\psi \in C^{\infty}_c(\R^+ \times \Omega)$:
\begin{equation*}
\int_{0}^{\infty} \int_{\Omega} \rho \psi_t + \rho u \cdot \nabla \psi \; dtdx = 0,
\end{equation*}
\begin{equation*}
\int_{0}^{\infty} \int_{\Omega} \rho \phi(u) \psi_t + (\rho \phi(u)u - (\sqrt{\rho}\mathbb{\MT}^s- 2\sqrt{ \rho}\mathbb{K})\phi'(u)) \cdot\nabla \psi\; dtdx = \left\langle R_{\phi},\psi\right\rangle 
\end{equation*}
with $\sqrt{\rho}\mathbb{\MT}^s= \rho D(u)$ and $\dive(\sqrt{\rho}\mathbb{K})= \rho \nabla \left( \frac{\Delta \sqrt{\rho}}{\sqrt{\rho}}\right)$.
\end{definition}

Let us observe that, taking a sequence of $\phi_n$ such that $\phi_n (y) \rightarrow y_i$ and $||\phi''||_{L^{\infty}} \rightarrow 0$ one can formally obtain $\eqref{qnsAPP}$. Indeed, choosing in suitable way the sequence of renormalized weak solutions, one can obtain in the limit exactly a finite energy weak solution in the sense of Definition \ref{def:1}. The following theorem holds.

\begin{theorem}\label{teo:3.2}
Let $(\re, \ue, V_\e)$ be a weak solution of $\eqref{1.8}$ in the sense of Definition \ref{def1app}. Let $\beta_{\delta}^l$ the truncation defined before. Then the following equalities hold:
\begin{itemize}
	\item For any $\psi \in C_c^{\infty}((0,T) \times \T; \R)$
	\begin{equation*}
	\begin{aligned}
& \int_{\T} \rho^0 \beta_{\delta}^l(u^0) \psi(0,x) dx 
+ \iint_{(0,t)\times\T} \re \beta_{\delta}^l(\ue) \partial_t\psi dx ds \\
& - \iint_{(0,t)\times\T} \re \ue \beta_{\delta}^l(\ue) \nabla \psi dx ds - 2\iint_{(0,t)\times\T} \re^{\gamma/2} \nabla \re^{\gamma/2} \nabla_y \beta_{\delta}^l(\ue) \psi dx ds \\
& - \iint_{(0,t)\times\T} \rre \mathcal{T_{\e}}^s : \nabla_y \beta_{\delta}^l(\ue) \otimes \nabla \psi dx ds \\
& + 2 \iint_{(0,t)\times\T} \rre \nabla^2 \rre : \nabla_y \beta_{\delta}^l(\ue) \otimes \nabla \psi dx ds \\
& - 2\iint_{(0,t)\times\T} \nabla \rre \otimes \nabla \rre : \nabla_y \beta_{\delta}^l(\ue) \otimes \nabla \psi dx ds \\
& + \iint_{(0,t)\times\T} \re \ue \nabla_y \beta_{\delta}^l(\ue) \psi dx ds
+ \iint_{(0,t)\times\T} \re \nabla V_{\e} \nabla_y \beta_{\delta}^l(\ue) \psi dx ds \\
& + \tilde{R}_{\e}^{\delta} + \bar{R}_{\e}^{\delta} = 0
\end{aligned}
	\end{equation*}
	where:
	\begin{equation*}
    \begin{aligned}
        \tilde{R}_{\e}^{\delta} =& - \e \iint_{(0,t) \times \T} \re |\ue|^2 \ue \nabla_y \beta_{\delta}^l(\ue)\psi dxds \\
        &- \e \iint_{(0,t) \times \T} \ue \nabla_y \beta_{\delta}^l(\ue)\psi dxds
    \end{aligned}
	\end{equation*}
	\begin{equation*}
	\begin{aligned}
	& \bar{R}_{\e}^{\delta} =  - \iint_{(0,T) \times \T} \mathcal{T} \mathcal{T}^s \nabla_y^2 \beta_{\delta}^l(\ue) \psi -\iint_{(0,T) \times \T} \mathcal{T} \nabla^2 \rre  \nabla_y^2 \beta_{\delta}^l(\ue) \psi \\
	& + \iint_{(0,T) \times \T}  \nabla \re^{1/4} \otimes \nabla \re^{1/4} \nabla_y^2 \beta_{\delta}^l(\ue) \mathcal{T} \psi \\
	\end{aligned}
	\end{equation*}
	\item For any $\phi \in C_c^{\infty}((0,T) \times \T; \R)$ the following equality holds:
	\begin{equation}\label{T}
	\begin{aligned}
	&\iint_{(0,T) \times \T} \rre \mathcal{T_{\e}} \hat{\beta_\delta}(\ue) \phi = - \iint_{(0,T) \times \T}  \hat{\beta_\delta}(\ue) \re \ue \otimes \nabla \phi \\
    & - \iint_{(0,T) \times \T} \rre \ue \phi \nabla_y \hat{\beta_\delta}(\ue) \mathcal{T_{\e}}  - 2 \iint_{(0,T) \times \T} \rre \ue \otimes \nabla \rre \phi \hat{\beta_\delta}(\ue) \end{aligned}
	\end{equation}
\end{itemize}
\end{theorem}
\begin{proof}
	The proof is well known and follows the arguments of \cite{AHS} and \cite{LV}; for this reason we omit the computations.
\end{proof}

\subsection{Proof of Theorem \ref{teo:3}}
In order to prove Theorem \ref{teo:3} one has to perform two vanishing limits, first the limit as $\epsilon \rightarrow 0$ and then the truncation parameter $\delta \rightarrow 0$. We collect the bounds indipendent of $\e$ coming from $\eqref{ene1}$ and $\eqref{bd1}$. There exists a constant $C >0$ indipendent on $\e$ such that:
\begin{equation}\label{4.1}
\begin{aligned}
& ||\rre \ue||_{L^{\infty}_t L^2_x} \leq C \; \; ||\re^{\gamma/2}||_{L^{\infty}_t L^2_x} \leq C\; \; ||\nabla \rre||_{L^{\infty}_t L^2_x} \leq C \\
& ||\mathcal{T_{\e}}||_{L^{2}_t L^2_x} \leq C \; \; ||\rre \ue||_{L^{2}_t L^2_x} \leq C \; \; ||\sqrt{4/\gamma}\nabla \re^{\gamma/2}||_{L^{2}_t L^2_x} \leq C\\
&||\rre \nabla^2 \log \re||_{L^{2}_t L^2_x} \leq C \; \; ||\partial_t \re||_{L^2_t L^1_x}\leq C 
\end{aligned}
\end{equation}
using Definition $\ref{def:1}$ with $r_0=r_1=\e$ one deduces also:
\begin{equation}\label{4.2}
\begin{aligned}
& ||\re \ue||_{L^{2}_t L^2_x} \leq C \; \; ||\nabla(\re \ue)||_{L^{2}_t L^{3/2}_x} \leq C
\end{aligned}
\end{equation}
and from $\eqref{ene1}$:
\begin{equation}\label{4.5}
\begin{aligned}
&||\sqrt{\e}\ue||_{L^{2}_t L^2_x} \leq C \; \; ||\e^{1/4} \re^{1/4} \ue||_{L^4_t L^4_x} \leq C
\end{aligned}
\end{equation}

\begin{lemma}[\textit{Convergence Lemma}]\label{lem: 4.1}
	Let $\lbrace (\re, \ue, V_{\e}) \rbrace$ be a sequence of weak solution of $\eqref{1.8}$ with data satisfying  $\eqref{1.3}$ and $\eqref{3.2}$. Then 
	\begin{itemize}
		\item Up to subsequences there exists $\rho,m,\mathcal{T}, V$ and $\Lambda$ such that
		\begin{equation}\label{5.4}
		\begin{aligned}
		& \re \rightarrow \rho \text{ strongly in } L^2((0,T);H^1(\T)) \cap C([0,T]; L^2(\T)) \\
		&\re \ue \rightarrow m \text{ strongly in } L^p((0,T);L^p(\T)) \text{ with } p \in [1,3) \\
		&\mathcal{T_{\e}} \weakto \mathcal{T} \text{ weakly in } L^2((0,T);L^2(\T)) \\
		&\rre \ue \rightharpoonup \Lambda \text{ weakly $\star$ in } L^{\infty}((0,T);L^2(\T)) \\
		&\nabla V_\e \rightarrow \nabla V \text{ strongly in } C((0,T);L^2(\T))
		\end{aligned}
		\end{equation}
		Moreover, $\Lambda$ is such that $ \sqrt{\rho}\Lambda = m$.
		\item The following convergences hold for the density:
		\begin{equation}\label{5.5}
		\begin{aligned}
		& \nabla \rre \rightharpoonup  \nabla \sqrt{\rho} \text{ weakly in } L^2((0,T);H^2(\T))\\
		& \re^{\gamma/2} \rightarrow \rho^{\gamma/2} \text{ strongly in } L^1((0,T) \times \T)\\
		&\nabla \re^{\gamma/2} \rightarrow \nabla \rho^{\gamma/2} \text{ weakly in }  L^2((0,T);L^2(\T)). 		
		\end{aligned}
		\end{equation}
	\end{itemize}
\end{lemma}

\begin{proof}
Since $\rho_\e \in L^2((0,T); W^{2,3/2}(\T))$ and 
$\partial_t \rho_\e \in L^2((0,T); W^{-1,2}(\T))$, an application of the 
Aubin--Lions lemma yields the strong convergence
\[
\rho_\e \to \rho \qquad \text{in } L^2((0,T); H^1(\T)).
\]
Moreover, since $\rho_\e$ is also bounded in 
$L^\infty((0,T); W^{1,3/2}(\T))$, another application of the Aubin--Lions lemma implies
\[
\rho_\e \to \rho \qquad \text{in } C([0,T]; L^2(\T)).
\]

Similarly, since 
$\rho_\e u_\e \in L^2((0,T); W^{1,3/2}(\T))$ and \\
$\partial_t(\rho_\e u_\e) \in L^2((0,T); W^{-1,3/2}(\T))$, the Aubin--Lions lemma gives the strong convergence stated in~\eqref{5.4}\(_2\).

The convergences in \eqref{5.4}\(_3\) and \eqref{5.4}\(_4\) follow from standard weak compactness arguments.  
The convergence in \eqref{5.4}\(_5\) is a direct consequence of \eqref{5.4}\(_1\) together with the weak convergence $g_\e \rightharpoonup g$ in $L^2(\T)$.

The identity $\sqrt{\rho}\,\Lambda = m$ follows from the strong convergence in \eqref{5.4}\(_1\) combined with the weak convergence in \eqref{5.4}\(_4\).

Finally, the convergences in \eqref{5.5} follow from \eqref{4.1}, \eqref{4.2}, and standard weak compactness considerations.

\end{proof}

\begin{lemma}\label{4.2sec3}
	Let $f \in C \cap L^{\infty}(\R^3;\R)$ and $(\re,\ue)$ be a solution of $\eqref{1.8}$ and let $u$ be defined as follows:
	\begin{equation}\label{4.14}
	u = \left \{ \begin{aligned}
	 \frac{m(t,x)}{\rho(t,x)} = \frac{\Lambda(t,x)}{\sqrt{\rho(t,x)}} & \qquad (t,x) \in \lbrace \rho > 0 \rbrace \\
	0 &\qquad  (t,x) \in \lbrace \rho = 0 \rbrace 
	\end{aligned} \right.
	\end{equation}
	Then the following convergence hold:
	\begin{equation}\label{5.7}
	\begin{aligned}
	& \re f(\ue) \rightarrow \rho f(u) \text{ strongly in } L^p((0,T)\times \T) \text{ for any } p<3 \\
	& \nabla \rre f(\ue) \rightarrow \nabla \sqrt{\rho} f(u) \text{ strongly in } L^p((0,T)\times \T) \text{ for any } p< 6 \\
	& \re \ue f(\ue) \rightarrow \rho u f(u) \text{ strongly in } L^p((0,T)\times \T) \text{ for any } p< 2\\
	& \re^{\gamma/2} f(\ue) \rightarrow \rho^{\gamma/2} u f(u) \text{ strongly in } L^p((0,T)\times \T) \text{ for any } p< 6 \\
	& \re \nabla V_\e f(\ue) \rightarrow \rho \nabla V f(u) \text{ strongly in } L^2((0,T) \times \T)
	\end{aligned}
	\end{equation}
\end{lemma}
\begin{proof}
From Lemma $\ref{lem: 4.1}$, up to subsequence we know that:
\begin{equation*}
\begin{aligned}
& \re \rightarrow \rho \text{ a.e in } (0,T) \times \T \\
&\re \ue \rightarrow m \text{ a.e in } (0,T) \times \T \\
& \nabla \re \rightarrow \nabla \rho \text{ a.e in } (0,T) \times \T
\end{aligned}
\end{equation*}
Moreover, by Fatou Lemma we have:
\begin{equation*}
\iint \liminf_{\e \rightarrow 0} \frac{m_{\e}^2}{\re} \leq \liminf_{\e \rightarrow 0} \iint \frac{m_{\e}^2}{\re} < \infty
\end{equation*}
therefore $m = 0$ on $\lbrace \rho =0 \rbrace$ and $\sqrt{\rho} u \in L^{\infty}((0,T), L^2(\T))$. Moreover $m = \rho u = \sqrt{\rho}\Lambda$.\\
On $\lbrace \rho > 0 \rbrace$ it holds:
\begin{equation*}
\re f(\ue) \rightarrow \rho f(u) \text{ a.e in } \lbrace \rho > 0 \rbrace,
\end{equation*}
since $f \in L^{\infty}(\R^3,\R)$ we have
\begin{equation*}
|\re f(\ue)| \leq |\re| ||f(\ue)||_{\infty} \rightarrow 0 \text{ a.e in } \lbrace \rho = 0 \rbrace,
\end{equation*}
then $\re f(\ue) \rightarrow \rho f(u)$ a.e in $(0,T) \times \T$ and the convergence $\eqref{5.7}_1$ follows from the uniform bound $||\re||_{L^{\infty}_t L^3_x} \leq C$ and Vitali Theorem. \\
In view of Lemma \ref{lem: 4.1} the same argument  show $\eqref{5.7}_3$, $\eqref{5.7}_4$, $\eqref{5.7}_5$. \\
Regarding $\eqref{5.7}_2$, from Lemma $\ref{lem: 4.1}$ we have that $\rho$ is Sobolev function then 
\begin{equation*}
\nabla \rho = 0  \text{ in } \lbrace \rho = 0 \rbrace
\end{equation*} 
see \cite{evans}. The proof is completed.
\end{proof}

Let $\{(\rho_\e, u_\e, V_\e)\}_{\e>0}$ be a sequence of weak solutions to system~\eqref{1.8} in the sense of Definition~\ref{def1app}.  
By Lemma~\ref{lem: 4.1}, there exist functions $\rho, m, \Lambda, V$ and $\mathcal{T}$ such that the convergences in~\eqref{5.4} and \eqref{5.5} hold.  
We define the limit velocity field $u$ as in Lemma~\ref{4.2sec3} and we obtain:

\begin{equation*}
\sqrt{\rho} u \in L^{\infty}((0,T); L^2(\T)) \quad  \mathcal{T} \in L^{2}((0,T); L^2(\T)) \quad m = \sqrt{\rho}\Lambda = \rho u. 
\end{equation*} 

By using $\eqref{5.4}_1$, $\eqref{5.4}_2$  and the weak formulation of continuity equation $\eqref{cont1}$ as $\e \rightarrow 0$ it holds:
\begin{equation*}
\begin{aligned}
    & \int_{\T} \re^0 \phi(0,x) dx + \iint_{(0,t)\times\T} \re \phi_t dx ds + \iint_{(0,t)\times\T} \re \ue \nabla \phi dx ds
\rightarrow \\
& \int_{\T} \rho^0 \phi(0,x) dx + \iint_{(0,t)\times\T} \rho \phi_t dx ds + \iint_{(0,t)\times\T} \rho u \nabla \phi dx ds
\end{aligned}
\end{equation*}
for any $\phi \in C_c^{\infty}((0,T) \times \T)$, that is \eqref{eq: cont}.\\ 

Now consider the momentum equations. Let $l \in \lbrace 1,2,3 \rbrace$ fixed, $\psi \in C_c^{\infty}((0,T) \times \T; \R)$, using Theorem $\ref{teo:3.2}$ the following equality hold:
\begin{equation}
\begin{aligned}
& \int_{\T} \rho^0 \beta_{\delta}^l(u^0)\psi(0,x) dx
+ \iint_{(0,t)\times\T} \re \beta_{\delta}^l(\ue)\partial_t \psi dx ds \\
& - \iint_{(0,t)\times\T} \re \ue \beta_{\delta}^l(\ue) \nabla \psi dx ds  - \iint_{(0,t)\times\T} \rre \mathcal{T_{\e}}^s : \nabla \beta_{\delta}^l(\ue) \otimes \nabla \psi dx ds \\
& - 2\iint_{(0,t)\times\T} \re^{\gamma/2} \nabla \re^{\gamma/2} \nabla_y \beta_{\delta}^l(\ue)\psi dx ds \\
& - \iint_{(0,t)\times\T} \rre \nabla^2 \rre : \nabla_y \beta_{\delta}^l(\ue) \otimes \nabla \psi dx ds \\
& + \iint_{(0,t)\times\T} \rre \nabla \re^{1/4} \otimes \nabla \re^{1/4} : \nabla_y \beta_{\delta}^l(\ue) \otimes \nabla \psi dx ds \\
& - \iint_{(0,t)\times\T} \mathcal{T_{\e}}\mathcal{T_{\e}}^s \nabla_y^2 \beta_{\delta}^l(\ue)\psi dx ds
- \iint_{(0,t)\times\T} \mathcal{T_{\e}}\nabla^2 \rre \nabla_y^2 \beta_{\delta}^l(\ue)\psi dx ds \\
& + \iint_{(0,t)\times\T} \nabla \re^{1/4} \otimes \nabla \re^{1/4}\nabla_y^2 \beta_{\delta}^l(\ue)\mathcal{T}\psi dx ds \\
& + \iint_{(0,t)\times\T} \re \ue \nabla_y \beta_{\delta}^l(\ue)\psi dx ds
- \e \iint_{(0,t)\times\T} \re |\ue|^2 \ue \nabla_y \beta_{\delta}^l(\ue)\psi dx ds \\
& - \e \iint_{(0,t)\times\T} \ue \nabla_y \beta_{\delta}^l(\ue)\psi dx ds = 0.
\end{aligned}
\end{equation}
First take the limit as $\e \rightarrow 0$ and $\delta>0$ fixed. Then:
\begin{equation*}
\iint_{(0,t)\times\T} \re \beta_{\delta}^l(\ue)\partial_t \psi dx ds
\;\rightarrow\;
\iint_{(0,t)\times\T} \rho \beta_{\delta}^l(u)\partial_t \psi dx ds
 \quad  \text{ using } \eqref{5.7}_1,
\end{equation*}
\begin{equation*}
 \iint_{(0,t)\times\T} \re \ue \beta_{\delta}^l(\ue)\nabla \psi dx ds
\;\rightarrow\;
\iint_{(0,t)\times\T} \rho u \beta_{\delta}^l(u)\nabla \psi dx ds
 \quad  \text{ using } \eqref{5.7}_3.
\end{equation*}
Using $\eqref{5.7}_1$ and $\eqref{5.4}_3$ we have, 
\begin{equation*}	
\iint_{(0,t)\times\T} \rre \mathcal{T_{\e}}^s : \nabla \beta_{\delta}^l(\ue) \otimes \nabla \psi dx ds
\rightarrow
\iint_{(0,t)\times\T} \sqrt{\rho}\,\mathcal{T}^s : \nabla \beta_{\delta}^l(u) \otimes \nabla \psi dx ds,
\end{equation*}
thanks to  $\eqref{5.7}_4$ and $\eqref{5.5}_3$ we have
\begin{equation*}
 \iint_{(0,t)\times\T} \re^{\gamma/2} \nabla \re^{\gamma/2} \nabla_y \beta_{\delta}^l(\ue)\psi dx ds
\;\rightarrow\;
\iint_{(0,t)\times\T} \rho^{\gamma/2} \nabla \rho^{\gamma/2} \nabla_y \beta_{\delta}^l(u)\psi dx ds,
\end{equation*}
thanks to $\eqref{5.7}_1$ and  $\eqref{5.5}_1$: 
\begin{equation*}
\begin{aligned}
    & \iint_{(0,t)\times\T} \rre \nabla^2 \rre : \nabla_y \beta_{\delta}^l(\ue) \otimes \nabla \psi dx ds\rightarrow \\
    & \iint_{(0,t)\times\T} \sqrt{\rho}\,\nabla^2 \sqrt{\rho} : \nabla_y \beta_{\delta}^l(u) \otimes \nabla \psi dx ds \\
& \iint_{(0,t)\times\T} \nabla \rre \otimes \nabla \rre : \nabla_y \beta_{\delta}^l(\ue) \otimes \nabla \psi dx ds \rightarrow\\ &\iint_{(0,t)\times\T} \nabla \sqrt{\rho} \otimes \nabla \sqrt{\rho} : \nabla_y \beta_{\delta}^l(u) \otimes \nabla \psi dx ds.
\end{aligned}
\end{equation*}
Using $\eqref{5.7}_3$ we deduce 
\begin{equation*}
 \iint_{(0,t)\times\T} \re \ue \nabla_y \beta_{\delta}^l(\ue)\psi dx ds
\;\rightarrow\;
\iint_{(0,t)\times\T} \rho u \nabla_y \beta_{\delta}^l(u)\psi dx ds 
\end{equation*}
\begin{equation*}
 \e \iint_{(0,t)\times\T} \re |\ue|^2 \ue \nabla_y \beta_{\delta}^l(\ue)\psi dxds \leq C_{\delta} \sqrt{\e} \rightarrow 0 \text{ as } \e \rightarrow 0 
\end{equation*}
\begin{equation*}
\begin{aligned}
    \e \iint_{(0,t)\times\T} \ue \nabla_y \beta_{\delta}^l(\ue)\psi dxds  & \leq C_{\delta}  \e^{1/4}||\re||^{1/4}_{L^1_{t,x}} ||\e^{1/4} \re^{1/4} \ue||^3_{L^4_{t,x}} \\
    &\leq C_{\delta}  \e^{1/4} \rightarrow 0 \text{ as } \e \rightarrow 0.
\end{aligned}
\end{equation*}
Finally:
\begin{equation*}
\begin{aligned}
    R_{\e}^\delta & := \iint_{(0,t)\times\T} \mathcal{T_{\e}}\mathcal{T_{\e}}^s \nabla_y^2 \beta_{\delta}^l(\ue)\psi dx ds \\
    & + \iint_{(0,t)\times\T} \mathcal{T_{\e}}\nabla^2 \rre \nabla_y^2 \beta_{\delta}^l(\ue)\psi dx ds \\
& + \iint_{(0,t)\times\T} \nabla \re^{1/4}\otimes\nabla \re^{1/4}\,\nabla_y^2 \beta_{\delta}^l(\ue)\mathcal{T}\psi dx ds
 \\
    & \leq  ||\mathcal{T_{\e}}||^2_{L^2_{t,x}} ||\nabla^2_y \beta_{\delta}^l(\ue)||_{L^{\infty}_{t,x}}  + ||\mathcal{T_{\e}}||_{L^2_{t,x}} ||\nabla^2 \rre||_{L^2_{t,x}} ||\nabla^2_y \beta_{\delta}^l(\ue)||_{L^{\infty}_{t,x}} \\
    & +  ||\mathcal{T_{\e}}||_{L^2_{t,x}} ||\nabla \re^{1/4}||^2_{L^4_{t,x}}  ||\nabla^2_y \beta_{\delta}^l(\ue)||_{L^{\infty}_{t,x}} \leq C \delta 
\end{aligned}
\end{equation*}
Therefore the triple $(\rho,u, V)$ satisfies:
\begin{equation}\label{4.24}
\begin{aligned}
& \iint_{(0,t)\times\T} \rho \beta_{\delta}^l(u)\partial_t \psi dx ds
- \iint_{(0,t)\times\T} \rho u \beta_{\delta}^l(u)\nabla \psi dx ds \\
& 
- \iint_{(0,t)\times\T} \sqrt{\rho}\,\mathcal{T}^s : \nabla \beta_{\delta}^l(u)\otimes\nabla \psi dx ds  - \iint_{(0,t)\times\T} \rho^{\gamma/2}\nabla \rho^{\gamma/2}\nabla_y \beta_{\delta}^l(u)\psi dx ds \\
&
- 2\iint_{(0,t)\times\T} \sqrt{\rho}\,\nabla \sqrt{\rho} : \nabla_y \beta_{\delta}^l(u)\otimes\nabla \psi dx ds \\
& + 2\iint_{(0,t)\times\T} \nabla \sqrt{\rho}\otimes\nabla \sqrt{\rho} : \nabla_y \beta_{\delta}^l(u)\otimes\nabla \psi dx ds \\
& + \iint_{(0,t)\times\T} \rho u \nabla_y \beta_{\delta}^l(u)\psi dx ds + \iint_{(0,t)\times\T} \rho \nabla V \beta_{\delta}^l(u)\psi dx ds \\
& - \int_{\T} \rho^0 \beta_{\delta}^l(u^0)\psi(0,x) dx
+ \langle \mu^{\delta}, \psi \rangle = 0
\end{aligned}
\end{equation}
where $\mu^{\delta}$ is a measure arising a weak limit of the remainder term $R_{\e}^\delta$ such that:
\begin{equation*}
R_{\e}^\delta \rightarrow \mu^{\delta} \text{ in } \mathcal{M}(\T,\mathbb{R}),
\end{equation*}
and its total variation satisfies:
\begin{equation*}
|\mu^{\delta}|(\T) \leq C \delta \rightarrow 0 \text{ as } \delta \rightarrow 0.
\end{equation*}
Thanks to Lemma $\ref{lemma: prope}$ and Dominated Convergence Theorem one concludes that $\eqref{4.24}$ converges to $\eqref{eq: mom}$ as $\delta \rightarrow 0$:
\begin{equation*}
\begin{aligned}
& \int_{\T} \rho^0 u^{l,0} \psi(0,x) dx
+ \iint_{(0,t)\times\T} \rho u^l \partial_t \psi dx ds
+ \iint_{(0,t)\times\T} \rho u u^l \nabla \psi dx ds \\
& - \iint_{(0,t)\times\T} \sqrt{\rho}\,\mathcal{T}^s_{i,j}\nabla_j \psi dx ds  - \iint_{(0,t)\times\T} \rho^{\gamma/2}\nabla_l \rho^{\gamma/2}\psi dx ds \\
&- 2\iint_{(0,t)\times\T} \sqrt{\rho}\,\nabla_{i,j}^2 \sqrt{\rho}\,\nabla \psi dx ds
+ 2\iint_{(0,t)\times\T} \nabla_i\sqrt{\rho}\,\nabla_j\sqrt{\rho}\,\nabla_i \psi dx ds \\
& + \iint_{(0,t)\times\T} \rho u \psi dx ds = 0
\end{aligned}
\end{equation*}
Thanks to $\eqref{5.4}_1$ and $\eqref{5.4}_5$ the Poisson equation is satisfied pointwise, namely:
\begin{equation*}
-\Delta V = \rho - g \quad \text{ a.e in }(0,T) \times \T.
\end{equation*}
To conclude the proof one has to verify $\eqref{eq: diss}$. Starting from $\eqref{T}$ in Theorem \ref{teo:3.2}, for any $\phi \in C_c^{\infty}((0,T) \times \T; \R)$ it holds:
\begin{equation*}
\begin{aligned}
\iint_{(0,t)\times\T} \rre \mathcal{T_{\e}} \hat{\beta_{\delta}}(\ue)\phi dx ds
&= \iint_{(0,t)\times\T} \hat{\beta_{\delta}}(\ue)\,\re \ue \times \nabla \phi dx ds \\
&\quad - \iint_{(0,t)\times\T} \rre \ue \phi \,\nabla_y \hat{\beta_{\delta}}(\ue)\mathcal{T_{\e}} dx ds \\
&\quad - 2 \iint_{(0,t)\times\T} \rre \ue \otimes \nabla \rre \,\phi\, \hat{\beta_{\delta}}(\ue) dx ds .
\end{aligned}
\end{equation*}
For fixed $\delta$, by using $\eqref{5.4}_3$ and $\eqref{5.7}_1$ as $\e \rightarrow 0$ one obtains:

\begin{equation*}
\begin{aligned}
& \iint_{(0,t)\times\T} \rre \mathcal{T_{\e}} \hat{\beta_{\delta}}(\ue)\phi dx ds
\;\rightarrow\;
\iint_{(0,t)\times\T} \sqrt{\rho}\,\mathcal{T}\,\hat{\beta_{\delta}}(u)\phi dx ds \\
& \iint_{(0,t)\times\T} \hat{\beta_{\delta}}(\ue)\re \ue \otimes \nabla \phi dx ds
\;\rightarrow\;
\iint_{(0,t)\times\T} \hat{\beta_{\delta}}(u)\,\rho u \otimes \nabla \phi dx ds
\end{aligned}
\end{equation*}
Using $\eqref{5.7}_1$ and the weak convergence of $\nabla \rre$ in $L^2_{t,x}$:
\begin{equation*}
\iint \rre \ue \otimes \nabla \rre \hat{\beta_{\delta}}(\ue) \rightarrow \iint \sqrt{\rho} u \otimes \nabla \sqrt{\rho} \hat{\beta_{\delta}}(u).
\end{equation*}
Now consider:
\begin{equation*}
\bar{R_{\e}}^{\delta} = \rre \ue \phi \nabla_y \hat{\beta_{\delta}}(\ue)\mathcal{T_{\e}},
\end{equation*}
by using $\eqref{4.1}$ and Lemma $\ref{lemma: prope}$ we have:
\begin{equation*}
||\bar{R_{\e}}^{\delta}||_{L^1_{t,x}} \leq C ||\rre \ue||_{L^{\infty}_t L^2_x} ||\mathcal{T_{\e}}||_{L^2_{t,x}} ||\nabla_y \hat{\beta_{\delta}}(\ue)||_{L^{\infty}_{t,x}} \leq C \delta,
\end{equation*}
there exists a measure $\bar{\mu}^{\delta}$ such that:
\begin{equation}\label{4.29}
\iint \bar{R_{\e}}^{\delta} \nabla \phi \rightarrow <\bar{\mu}^{\delta}, \nabla \phi>,  
\end{equation}
and its total variation satisfies:
\begin{equation*}
|\bar{\mu}^{\delta}|(\T) \leq C \delta.
\end{equation*}
Collecting all the previous terms:
\begin{equation*}
\begin{aligned}
\iint_{(0,t)\times\T} \sqrt{\rho}\,\mathcal{T}\,\hat{\beta_{\delta}}(u)\phi dx ds
&= - \iint_{(0,t)\times\T} \hat{\beta_{\delta}}(u)\,\rho u \otimes \nabla \phi dx ds \\
&\quad - 2 \iint_{(0,t)\times\T} \sqrt{\rho}\,u \otimes \nabla \sqrt{\rho}\,\hat{\beta_{\delta}}(u) dx ds \\
&\quad - \langle \bar{\mu}^{\delta}, \nabla \phi \rangle .
\end{aligned}
\end{equation*}
Using once again Lemma $\ref{lemma: prope}$, $\eqref{4.29}$ and Dominated convergence theorem one gets exactly $\eqref{eq: diss}$. As usual, the energy inequalities $\eqref{eq: ene ineq}$ and $\eqref{eq: bd3}$ follow from the lower semicontinuity of the norms. The proof of Theorem \ref{teo:3} is now completed.

\begin{acknowledgments}
The author would like to acknowledge partial support from INdAM–GNAMPA and from the Italian research projects PRIN 2020 “Nonlinear evolution PDEs, fluid dynamics and transport equations: theoretical foundations and applications” and PRIN 2022 “Classical equations of compressible fluid mechanics: existence and properties of non-classical solutions.” The Author thank Prof. Corrado Lattanzio and Prof. Stefano Spirito for their helpful comments and guidance.
\end{acknowledgments}

\end{document}